\documentclass{article}%
\usepackage{amsmath}
\usepackage{amsfonts}
\usepackage{amssymb}
\usepackage{graphicx}
\usepackage{color}%
\providecommand{\U}[1]{\protect\rule{.1in}{.1in}}
\newtheorem{theorem}{Theorem}

\newtheorem{definition}[theorem]{Definition}

\newtheorem{lemma}[theorem]{Lemma}

\newtheorem{proposition}[theorem]{Proposition}
\newtheorem{remark}[theorem]{Remark}

\newenvironment{proof}[1][Proof]{\noindent\textbf{#1.} }{\ \rule{0.5em}{0.5em}}
\allowdisplaybreaks
\begin{document}

\title{On axially symmetric one phase free boundary problem in $\mathbb{R}^{n}$}
\author{Yong Liu\\School of Mathematics and Physics, \\North China Electric Power University, Beijing, China,\\Email: liuyong@ncepu.edu.cn
\and Kelei Wang\\School of Mathematics and Statistics, Wuhan University\\Email:wangkelei@whu.edu.cn
\and Juncheng Wei\\Department of Mathematics, \\University of British Columbia, Vancouver, B.C., Canada, V6T 1Z2\\Email: jcwei@math.ubc.ca}
\maketitle

\begin{abstract}
We construct a smooth axially symmetric solution to the classical one phase
free boundary problem in $\mathbb{R}^{n}$, $n\geq3.$ Its free boundary is of
\textquotedblleft catenoid\textquotedblright\ type. This is a higher
dimensional analogy of the Hauswirth-Helein-Pacard solution\cite{Pacard} in
$\mathbb{R}^{2}$. The existence of such solution is conjectured in
\cite[Remark 2.4]{Pacard}.

\end{abstract}

\section{Introduction and main results\label{Sec 1}}

Free boundary problems arise as mathematical models in many different
contexts, e.g., heat conduction, interface dynamics, evolution of ecological
systems. In this paper, we are interested in constructing new solutions for
the following classical one phase free boundary problem in the whole space:%
\begin{equation}
\left\{
\begin{array}
[c]{l}%
\Delta u=0\text{ in }\Omega:=\left\{  u>0\right\}  \subset\mathbb{R}^{n},\\
\left\vert \nabla u\right\vert =1\text{ on }\partial\Omega.
\end{array}
\right.  \label{free}%
\end{equation}
Here $\partial\Omega$ is the free boundary. The regularity theory of
(\ref{free}) has been studied for a long time, see for instances
\cite{C2,Ca1,Ca2,Ca3,Cs,Jerison3,Jerison6}. In the literature, the domain
$\Omega$ in the one phase problem is called \emph{exceptional domain} and the
function $u$ is called \emph{roof function}.

\medskip

The simplest solution to $\left(  \ref{free}\right)  $ is the one-dimensional
solution $x_{n}^{+}.$ This solution is unbounded, which constitutes a major
difficulty for the construction of other solutions using this one dimensional
profile. Another class of solutions of $\left(  \ref{free}\right)  $ is the
cone type solutions (homogeneous functions of degree one). Consider the
Alt-Caffarelli cone in $\mathbb{R}^{n}$ given by
\begin{equation}
\left\vert x_{n}\right\vert <\alpha_{n}\sqrt{x_{1}^{2}+...+x_{n}^{2}}.
\label{cone}%
\end{equation}
It is known that there exists a unique dimensional constant $\alpha_{n}$ (see
\cite{Alt,C1}) such that there is a solution to $\left(  \ref{free}\right)  $
whose free boundary is exactly this cone. It has been proved \cite{Jerison2}
that in dimension $n=7$ (actually also for $n=9,11,13,15$), the solution to
$\left(  \ref{free}\right)  $ corresponding to the cone $\left(
\ref{cone}\right)  $ is a minimizer for the energy functional%
\begin{equation}
J_{0}\left(  u\right)  :=\int\left[  \left\vert \nabla u\right\vert ^{2}%
+\chi_{\left(  0,+\infty\right)  }\left(  u\right)  \right]  .
\end{equation}
For a discussion on the existence and stability of more general cones other
than $\left(  \ref{cone}\right)  ,$ we refer to \cite{Hong,Jerison1}. We
remark that for a cone type solution which is also a minimizer of the energy
functional, it is expected that there should be a family of smooth solutions
to $\left(  \ref{free}\right)  $ whose free boundary is smooth and asymptotic
to the cone.

\medskip

We notice that the cone solution has a singularity at the origin. So far the
only nontrivial smooth solution with simply connected phase we know of is the
so-called Hauswirth-Helein-Pacard solution\cite{Pacard} in the plane(also
called hairpin solution\cite{Jerison5}). To describe this solution, we use
$\Phi$ to denote the map
\[
\left(  x,y\right)  \rightarrow\left(  x+\cos y\sinh x,y+\sin y\cosh x\right)
.
\]
Let $\Omega$ be the image of the region $\left\{  \left(  x,y\right)
:\left\vert y\right\vert <\frac{\pi}{2}\right\}  $ under this map. One checks
directly that
\[
\Omega=\left\{  \left(  x,y\right)  :\left\vert y\right\vert <\frac{\pi}%
{2}+\cosh x\right\}  .
\]
Let $u\left(  x,y\right)  =\cos y\cosh x.$ Then the function
\begin{equation}
U\left(  x,y\right)  =u\circ\Phi^{-1}\left(  x,y\right)  \label{PH}%
\end{equation}
is a solution to $\left(  \ref{free}\right)  $. It turns out that the
Hauswirth-Helein-Pacard solution plays an important role in the analysis of
other solutions of the one phase free boundary problem in the unit disk with
simply connected phase \cite{Jerison5}, similar to the role played by the
catenoids in the minimal surface theory \cite{Colding}.

\medskip

Using complex function theory, Traizet \cite{Traizet} (see also
\cite{Traizet2} for related results) established a one-to-one correspondence
between solutions to $\left(  \ref{free}\right)  $ and a special class of
minimal surfaces in $\mathbb{R}^{3}.$ Under this correspondence, $U$ is
transformed to the catenoid, a classical minimal surface. It also has been
proved there that $U$ is the unique (up to a scaling and the trivial one
dimensional solution) solution in $\mathbb{R}^{2}$ with simply connected phase
(see also \cite{KLT,Ruiz}). Unfortunately the correspondence between one phase
problem and minimal surface is not available in dimensions $n\geq3.$ However,
it is conjectured in \cite[Remark 2.4]{Pacard} that the
Hauswirth-Helein-Pacard solution should still have higher dimensional analogy.
In this paper, we confirm this conjecture.

\medskip

To state our result, we use $\left(  x_{1},...,x_{n-1},z\right)  $ to denote
the coordinate of $\mathbb{R}^{n}$ and set $r=\sqrt{x_{1}^{2}+...+x_{n-1}^{2}%
}.$

\begin{theorem}
\label{main}There exists a solution $u$ to $\left(  \ref{free}\right)  $
satisfying the following properties:\newline(I) $u$ depends only on $r$ and
$\left\vert z\right\vert .$ \newline(II) The positive phase $\Omega:=\left\{
\left(  x_{1},...,x_{n-1},z\right)  \in\mathbb{R}^{n}:u\left(  x_{1}%
,...,x_{n-1},z\right)  >0\right\}  $ can be described by
\[
\mathbb{R}^{n}\backslash\left\{  \left(  x_{1},...,x_{n-1},z\right)
:\left\vert z\right\vert <g\left(  r\right)  \right\}  ,
\]
for a function $g$ with $g\left(  1\right)  =0$ and
\begin{equation}
\lim_{r\rightarrow+\infty}\left(  g^{\prime}\left(  r\right)  r^{n-2}\right)
\in\lbrack0,+\infty). \label{lim}%
\end{equation}
\newline(III) In $\Omega,$ $\partial_{z}u>0$ for $z>0,$ and $\partial_{r}u<0$
for $r>0.$
\end{theorem}

\begin{remark}
Due to the scaling invariance of the problem, actually we have a family of
solutions $\frac{u\left(  \rho X\right)  }{\rho}$ with $\rho>0$ being a
parameter. It is to be expected that there should exist another two families
of axially symmetric solutions whose free boundaries are asymptotic to the
Alt-Caffarelli cone. The positive phases should have the form $\left\{
\left(  x_{1},...,x_{n-1},z\right)  :\left\vert z\right\vert <h\left(
r\right)  \right\}  ,$ where $h$ is a positive monotone function defined on
$[0,+\infty)$ for the first family of solutions, while $h$ is monotone and
defined on $[1,+\infty)$ with $h\left(  1\right)  =0$ for one of the solutions
in the second family.
\end{remark}

\medskip

Now let us describe the main difficulties and steps of the proof of Theorem
\ref{main} in the case of $n=3.$ The other cases are similar. A solution of
the one phase free boundary problem is \textit{formally} a critical point of
the energy functional $J_{0}.$ Given suitable boundary conditions, while it is
relatively easy to use minimizing arguments to obtain \textit{minimizers} (see
\cite{Alt}), variational methods in general are not directly applicable for
unstable critical points of $J_{0},$ due to the fact that $J_{0}$ is
\textbf{not} differentiable in usual functional spaces. Furthermore for the
solutions we are interested in this paper, they are indeed not minimizers.
Actually, in $\mathbb{R}^{3},$ it is conjectured that any minimizer should be
a trivial solution. (See \cite{C1} and \cite{Jerison1}.) Another difficulty we
are facing is that usually the solutions are unbounded and the energy is
actually equal to infinity.

\medskip

To overcome these difficulties, we proceed the proofs in two steps. In the
first step, for each fixed $k>0$, we construct two-end solutions to the
following two-component free boundary problem
\begin{equation}
\left\{
\begin{array}
[c]{l}%
\Delta u=0\text{ in }\left\{  \left\vert u\right\vert <1\right\}  ,\\
\left\vert \nabla u\right\vert =1\text{ on }\partial\left\{  \left\vert
u\right\vert <1\right\}  ,
\end{array}
\right.  \label{EN}%
\end{equation}
where the nodal set $\{u=0\}$ behaves like $\{|z|=k\log r+b\}$. The solutions
to problem (\ref{EN}) are bounded and relatively easier to deal with, though
we still need to overcome the problem of nonsmooth profiles and regularity
issues since the solutions are not minimizers and are of mountain-pass type in
terms of the new energy functional
\[
J_{1}\left(  u\right)  :=\int\left[  \left\vert \nabla u\right\vert ^{2}%
+\chi_{\left(  -1,+1\right)  }\left(  u\right)  \right]  .
\]
In the second step, we show that the solutions to (\ref{EN}), after some
rescaling, as $k\rightarrow0^{+}$, approach to a nontrivial solution of
(\ref{free}).

\medskip

The paper is organized as follows. From Section \ref{Sec2} to Section
\ref{Sec5}, we prove Theorem \ref{main} in the case $n=3.$ Then in Section
\ref{Sec6} we indicate the necessary modifications needed for general
$n\geq3.$ In Section \ref{Sec2}, we consider a family of regularized problems
and use variational arguments to show the existence of mountain pass type
solutions $U_{\varepsilon,a}$ in bounded domain $\Omega_{a}$. In Section
\ref{Sec3}, we prove that as $\varepsilon$ tends to zero, these mountain pass
solutions converge to a solution $V_{a}$ of \eqref{EN} in $\Omega_{a}$. We
also show the regularity of the free boundary of $V_{a}.$ In Section
\ref{Sec4}, we enlarge the domain by sending $a$ to infinity, and get a
solution $W_{k}$ for \eqref{EN} with prescribed asymptotic behavior at
infinity (nodal set looks like $k\ln r+b$). In Section \ref{Sec5}, we analyze
the precise asymptotic behavior of $W_{k}.$ Then by sending $k$ to zero, we
show that suitable blow up sequence of $W_{k}$ near the origin converges to a
solution of $\left(  \ref{free}\right)  .$ In the last section, we consider
the general case of $n>3.$

\bigskip

\noindent\textit{Acknowledgement.} The research of J. Wei is partially
supported by NSERC of Canada. Part of the paper was finished while Y. Liu was
visiting the University of British Columbia in 2016. He appreciates the
institution for its hospitality and financial support. K. Wang is supported by
\textquotedblleft the Fundamental Research Funds for the Central
Universities". We thank Professor N. Kamburov for pointing out \cite{KLT} to us.

\section{\bigskip Mountain pass solutions for a family of regularized
problems\label{Sec2}}

As we already mentioned in Section \ref{Sec 1}, there are three main
difficulties in dealing with problem (\ref{free}): firstly the energy
functional is not smooth, secondly the solution we are interested in is not a
minimizer and thirdly the solution is unbounded.

To overcome the above mentioned difficulties, we shall regularize the
functional $J_{1}$ and consider a family of smooth potentials $F_{\varepsilon
}\ $which approximates the characteristic function $\chi_{\left(  -1,1\right)
}\left(  \cdot\right)  $ of the interval $\left(  -1,1\right)  $.
$F_{\varepsilon}$ is defined in the following way. Let $\bar{F}$ be a smooth
monotone increasing function in $[0,+\infty)$ such that
\[
\bar{F}\left(  s\right)  =\left\{
\begin{array}
[c]{l}%
s^{2},s\in\left[  0,\frac{1}{2}\right]  ,\\
1-e^{-s},s\in\left(  1,+\infty\right)  .
\end{array}
\right.
\]
We may also assume that $\bar{F}^{\prime\prime}<0$ in $\left(  \frac{3}%
{4},+\infty\right)  .$ It is worth pointing out that the idea of regularizing
the potentials has been explored in some other related contexts, for instances
\cite{BCN,Wo}.

Let $\rho\geq0$ be a cutoff function satisfying $\rho\left(  s\right)
+\rho\left(  -s\right)  =1$ and
\[
\rho\left(  s\right)  =\left\{
\begin{array}
[c]{l}%
1,s<-\frac{1}{2},\\
0,s>\frac{1}{2}.
\end{array}
\right.
\]
For each $\varepsilon>0$ small, we define a smooth even potential
$F_{\varepsilon}$ on the interval $\left[  -1,1\right]  ,$ monotone increasing
in $\left[  -1,0\right]  ,$ to be
\[
F_{\varepsilon}\left(  s\right)  =\rho\left(  s\right)  \bar{F}\left(
\frac{s+1}{\varepsilon}\right)  +\left(  1-\rho\left(  s\right)  \right)
\bar{F}\left(  \frac{-s+1}{\varepsilon}\right)  .
\]
With this definition, $F_{\varepsilon}\leq1,$ and on any compact subinterval
of $\left(  -1,1\right)  ,$ $F_{\varepsilon}\rightarrow1,$ as $\varepsilon
\rightarrow0.$ We also have
\[
F_{\varepsilon}^{\prime\prime}\left(  \pm1\right)  =\frac{2}{\varepsilon^{2}%
}\rightarrow+\infty,\text{ \ as }\varepsilon\rightarrow0.
\]
Then instead of $J_{1}$, we shall consider the regularized functional
\[
\int\left[  \left\vert \nabla u\right\vert ^{2}+F_{\varepsilon}\left(
u\right)  \right]  .
\]
Note that $F_{\varepsilon}$ is a double well type potential and a critical
point of this functional solves the equation%
\begin{equation}
-\Delta u+\frac{1}{2}F_{\varepsilon}^{\prime}\left(  u\right)  =0.
\label{fepsilon}%
\end{equation}

Let $H_{\varepsilon}$ be the heteroclinic solution of the ODE%
\[
H_{\varepsilon}^{\prime\prime}=\frac{1}{2}F_{\varepsilon}^{\prime}\left(
H_{\varepsilon}\right)  ,
\]
with
\[
H_{\varepsilon}\left(  0\right)  =0, \quad H_{\varepsilon}\left(  \pm
\infty\right)  =\pm1.
\]
Heuristically, as $\varepsilon\rightarrow0,$ $H_{\varepsilon}$ converges to
the function
\[
\mathcal{H}\left(  x\right)  =\left\{
\begin{array}
[c]{l}%
x,\text{ \ for }-1<x<1,\\
1,\text{ \ for }x\geq1,\\
-1,\text{ for }x\leq-1.
\end{array}
\right.
\]
This is a one-dimensional solution of the following free boundary problem
\begin{equation}
\left\{
\begin{array}
[c]{l}%
\Delta u=0\text{ in }\Omega:=\left\{  \left\vert u\right\vert <1\right\}
\subset\mathbb{R}^{n},\\
\left\vert \nabla u\right\vert =1\text{ on }\partial\Omega.
\end{array}
\right.  \label{Int}%
\end{equation}
The existence and classification of solutions to this problem has been studied
in \cite{Kam,Liu2,Wangke}.

We would like to construct mountain pass type solutions for $\left(
\ref{fepsilon}\right)  $ on bounded domains with suitable boundary data, using
similar ideas as that of \cite{Gui}, where Morse index one solutions to the
Allen-Cahn equation are constructed.

To describe the boundary data, we need to know the asymptotic behavior of
$H_{\varepsilon}$ as $\varepsilon$ goes to zero.

\begin{lemma}
\label{h1}Let $\varepsilon>0$ be small. For any $x\in\left[  0,1+\varepsilon
\ln\varepsilon\right]  ,$ there holds%
\[
\left(  1-\varepsilon\right)  x\leq H_{\varepsilon}\left(  x\right)  \leq x.
\]

\end{lemma}

\begin{proof}
$H_{\varepsilon}$ satisfies
\begin{equation}
H_{\varepsilon}^{\prime2}=F_{\varepsilon}\left(  H_{\varepsilon}\right)
\leq1. \label{s}%
\end{equation}
Hence $H_{\varepsilon}\left(  x\right)  \leq x,$ for $x>0.$ In particular, in
the interval $\left[  0,1+\varepsilon\ln\varepsilon\right]  ,$
\begin{equation}
H\left(  x\right)  \leq1+\varepsilon\ln\varepsilon. \label{s3}%
\end{equation}

By the definition of $F_{\varepsilon},$
\[
1-F_{\varepsilon}\left(  s\right)  =1-\rho\left(  s\right)  \bar{F}\left(
\frac{s+1}{\varepsilon}\right)  -\left(  1-\rho\left(  s\right)  \right)
\bar{F}\left(  \frac{-s+1}{\varepsilon}\right)  .
\]
If $0\leq s<1+\varepsilon\ln\varepsilon,$ then $\bar{F}\left(  \frac
{s+1}{\varepsilon}\right)  =1-e^{-\frac{s+1}{\varepsilon}},\bar{F}\left(
\frac{-s+1}{\varepsilon}\right)  =1-e^{-\frac{-s+1}{\varepsilon}}.$ Hence
\begin{align*}
1-F_{\varepsilon}\left(  s\right)   &  =1-\rho\left(  s\right)  \left(
1-e^{-\frac{s+1}{\varepsilon}}\right)  -\left(  1-\rho\left(  s\right)
\right)  \left(  1-e^{-\frac{-s+1}{\varepsilon}}\right) \\
&  =\rho\left(  s\right)  e^{-\frac{s+1}{\varepsilon}}+\left(  1-\rho\left(
s\right)  \right)  e^{-\frac{-s+1}{\varepsilon}}.
\end{align*}
Consequently,
\[
1-F_{\varepsilon}\left(  s\right)  \leq\varepsilon.
\]
This together with $\left(  \ref{s3}\right)  $ implies that in the interval
$\left[  0,1+\varepsilon\ln\varepsilon\right]  ,$
\[
H_{\varepsilon}^{\prime}=\sqrt{F_{\varepsilon}\left(  H_{\varepsilon}\right)
}\geq1-\varepsilon,
\]
provided that $\varepsilon$ is small. The conclusion of the lemma then follows
from $H_{\varepsilon}\left(  0\right)  =0.$
\end{proof}

We use $t_{\varepsilon}$ to denote the point where
\begin{equation}
H_{\varepsilon}\left(  t_{\varepsilon}\right)  =1-\frac{\varepsilon}{2}.
\label{tstar}%
\end{equation}
By $\left(  \ref{s}\right)  ,$ $H_{\varepsilon}^{\prime}\left(  t_{\varepsilon
}\right)  =\frac{1}{2}$ and $H_{\varepsilon}^{\prime\prime}\left(
t_{\varepsilon}\right)  =\frac{1}{2\varepsilon}.$ Additionally,
$t_{\varepsilon}\in\left[  1+\varepsilon\ln\varepsilon,1\right]  .$

\begin{lemma}
\label{h2}For $t\in\lbrack t_{\varepsilon},+\infty),$
\[
H_{\varepsilon}\left(  t\right)  =1-\frac{\varepsilon}{2}e^{\frac
{t_{\varepsilon}}{\varepsilon}}e^{-\frac{t}{\varepsilon}}.
\]

\end{lemma}

\begin{proof}
Let $t\geq t_{\varepsilon}.$ Since $H_{\varepsilon}$ is monotone increasing,
using Lemma \ref{h1}, we find that $H_{\varepsilon}\left(  t\right)
\geq1-\frac{\varepsilon}{2}.$ It follows that%
\[
F_{\varepsilon}\left(  H_{\varepsilon}\right)  =\frac{\left(  1-H_{\varepsilon
}\right)  ^{2}}{\varepsilon^{2}}.
\]
Hence by $\left(  \ref{s}\right)  ,$
\[
H_{\varepsilon}^{\prime}=\frac{1-H_{\varepsilon}}{\varepsilon}.
\]
Consequently, $H_{\varepsilon}\left(  t\right)  =1-c_{\varepsilon}e^{-\frac
{t}{\varepsilon}}.$ It then follows from $\left(  \ref{tstar}\right)  $ that
\[
1-c_{\varepsilon}e^{-\frac{t_{\varepsilon}}{\varepsilon}}=1-\frac{\varepsilon
}{2}.
\]
This yields $c_{\varepsilon}=\frac{\varepsilon}{2}e^{\frac{t_{\varepsilon}%
}{\varepsilon}}.$ The proof is completed.
\end{proof}

Lemma \ref{h1} and Lemma \ref{h2}, together with the fact that $\left\vert
H_{\varepsilon}^{\prime}\right\vert \leq1,$ imply that $H_{\varepsilon}\left(
s\right)  -s\rightarrow0$ in $C^{0,\alpha}\left(  \left[  -1,1\right]
\right)  .$

Let $l>2$ be a large constant and $\delta_{\varepsilon}=O\left(
\varepsilon^{\frac{4}{3}}\right)  $ be the constant satisfying
\begin{equation}
H_{\varepsilon}\left(  2\right)  +H_{\varepsilon}^{\prime}\left(  2\right)
\delta_{\varepsilon}+\frac{1}{2}H_{\varepsilon}^{\prime\prime}\left(
2\right)  \delta_{\varepsilon}^{2}+\frac{1-H_{\varepsilon}\left(  2\right)
}{\varepsilon^{4}}\delta_{\varepsilon}^{3}=1. \label{delta}%
\end{equation}
Note that for $\varepsilon$ small, we have
\[
H_{\varepsilon}^{\prime}\left(  2\right)  =\frac{1-H_{\varepsilon}\left(
2\right)  }{\varepsilon}=\varepsilon H_{\varepsilon}^{\prime\prime}\left(
2\right)  .
\]
We define a family of $C^{2}$ functions
\[
w_{\varepsilon,l}\left(  x\right)  :=\left\{
\begin{array}
[c]{l}%
H_{\varepsilon}\left(  x\right)  ,x\in\left[  -l,2\right]  ,\\
H_{\varepsilon}\left(  2\right)  +H_{\varepsilon}^{\prime}\left(  2\right)
\left(  x-2\right)  +\frac{1}{2}H_{\varepsilon}^{\prime\prime}\left(
2\right)  \left(  x-2\right)  ^{2}+\frac{1-H_{\varepsilon}\left(  2\right)
}{\varepsilon^{4}}\left(  x-2\right)  ^{3},x\in\left[  2,2+\delta
_{\varepsilon}\right]  ,\\
-H_{\varepsilon}\left(  l\right)  +H_{\varepsilon}^{\prime}\left(  l\right)
\left(  x+l\right)  -\frac{1}{2}H_{\varepsilon}^{\prime\prime}\left(
l\right)  \left(  x+l\right)  ^{2},x\in\left[  -l-\varepsilon,-l\right]  .
\end{array}
\right.
\]
Observe that for $\varepsilon$ small enough, $H_{\varepsilon}^{\prime}\left(
l\right)  =-\varepsilon H_{\varepsilon}^{\prime\prime}\left(  l\right)  .$
Hence $w_{\varepsilon,l}^{\prime}\left(  -l-\varepsilon\right)  =0.$ Moreover,
we have $w_{\varepsilon,l}^{\prime}\left(  x\right)  \geq0.$

\begin{lemma}
\label{sub}$w_{\varepsilon,l}$ is a subsolution:%
\[
-w_{\varepsilon,l}^{\prime\prime}+\frac{1}{2}F_{\varepsilon}^{\prime}\left(
w_{\varepsilon,l}\right)  \leq0\text{, for }x\in\left[  -l-\varepsilon
,2+\delta_{\varepsilon}\right]  .
\]

\end{lemma}

\begin{proof}
We first prove this in the interval $\left[  2,2+\delta_{\varepsilon}\right]
.$ For $s\in\left[  1-\frac{\varepsilon}{2},1\right]  ,$
\[
F_{\varepsilon}\left(  s\right)  =\varepsilon^{-2}\left(  1-s\right)
^{2},\text{ }F_{\varepsilon}^{\prime}\left(  s\right)  =2\varepsilon
^{-2}\left(  s-1\right)  .
\]
It follows that
\begin{align*}
\frac{1}{2}F_{\varepsilon}^{\prime}\left(  w_{\varepsilon,l}\left(  x\right)
\right)   &  =\varepsilon^{-2}\left(  w_{\varepsilon,l}\left(  x\right)
-1\right) \\
&  =\varepsilon^{-2}\left[  H_{\varepsilon}\left(  2\right)  -1+H_{\varepsilon
}^{\prime}\left(  2\right)  \left(  x-2\right)  +\frac{1}{2}H_{\varepsilon
}^{\prime\prime}\left(  2\right)  \left(  x-2\right)  ^{2}+a\left(
x-2\right)  ^{3}\right]  .
\end{align*}
On the other hand, we compute%
\[
w_{\varepsilon,l}^{\prime\prime}\left(  x\right)  =H_{\varepsilon}%
^{\prime\prime}\left(  2\right)  +\frac{1-H_{\varepsilon}\left(  2\right)
}{\varepsilon^{4}}6\left(  x-2\right)  .
\]
Then using the fact that $\delta_{\varepsilon}=O\left(  \varepsilon^{\frac
{4}{3}}\right)  ,$ we find that for $x\in\left[  2,2+\delta_{\varepsilon
}\right]  ,$
\begin{align*}
&  -w_{\varepsilon,l}^{\prime\prime}+\frac{1}{2}F_{\varepsilon}^{\prime
}\left(  w_{\varepsilon,l}\right) \\
&  =-H_{\varepsilon}^{\prime\prime}\left(  2\right)  -\frac{1-H_{\varepsilon
}\left(  2\right)  }{\varepsilon^{4}}6\left(  x-2\right) \\
&  +\varepsilon^{-2}\left[  H_{\varepsilon}\left(  2\right)  -1+H_{\varepsilon
}^{\prime}\left(  2\right)  \left(  x-2\right)  +\frac{1}{2}H_{\varepsilon
}^{\prime\prime}\left(  2\right)  \left(  x-2\right)  ^{2}+\frac
{1-H_{\varepsilon}\left(  2\right)  }{\varepsilon^{4}}\left(  x-2\right)
^{3}\right] \\
&  =\left(  x-2\right)  \left[  -6\frac{1-H_{\varepsilon}\left(  2\right)
}{\varepsilon^{4}}+\varepsilon^{-2}\left(  H_{\varepsilon}^{\prime}\left(
2\right)  +\frac{1}{2}H_{\varepsilon}^{\prime\prime}\left(  2\right)  \left(
x-2\right)  +\frac{1-H_{\varepsilon}\left(  2\right)  }{\varepsilon^{4}%
}\left(  x-2\right)  ^{2}\right)  \right] \\
&  \leq0,
\end{align*}
provided that $\varepsilon$ is small enough.

Next we consider the case of $x\in\left[  -l-\varepsilon,-l\right]  .$ In this
case, we have
\begin{align*}
\frac{1}{2}F_{\varepsilon}^{\prime}\left(  w_{\varepsilon,l}\left(  x\right)
\right)   &  =\varepsilon^{-2}\left(  w_{\varepsilon,l}\left(  x\right)
+1\right) \\
&  =\varepsilon^{-2}\left[  1-H_{\varepsilon}\left(  l\right)  +H_{\varepsilon
}^{\prime}\left(  l\right)  \left(  x+l\right)  -\frac{1}{2}H_{\varepsilon
}^{\prime\prime}\left(  l\right)  \left(  x+l\right)  ^{2}\right]  .
\end{align*}
Moreover, $w_{\varepsilon,l}^{\prime\prime}\left(  x\right)  =-H_{\varepsilon
}^{\prime\prime}\left(  l\right)  .$ Then using the fact that $H^{\prime
\prime}\left(  l\right)  =\varepsilon^{-2}\left(  H_{\varepsilon}\left(
l\right)  -1\right)  ,$ we get
\begin{align*}
-w_{\varepsilon,l}^{\prime\prime}\left(  x\right)  +\frac{1}{2}F_{\varepsilon
}^{\prime}\left(  w_{\varepsilon,l}\left(  x\right)  \right)   &
=\varepsilon^{-2}\left[  H_{\varepsilon}^{\prime}\left(  l\right)  \left(
x+l\right)  -\frac{1}{2}H_{\varepsilon}^{\prime\prime}\left(  l\right)
\left(  x+l\right)  ^{2}\right] \\
&  =\varepsilon^{-2}H_{\varepsilon}^{\prime}\left(  l\right)  \left(
x+l\right)  \left(  1+\frac{x+l}{2\varepsilon}\right)  \leq0.
\end{align*}
The proof is finished.
\end{proof}

As we mentioned before, from Section \ref{Sec2} to Section \ref{Sec5}, we will
deal with the case of dimension $n=3.$ Recall that in the coordinate $\left(
r,z\right)  ,$ where $r=\sqrt{x_{1}^{2}+x_{2}^{2}},$ the catenoids are given
by $\epsilon r=\cosh\left(  \epsilon z\right)  ,$ with $\epsilon>0$ being a
parameter. They are classical minimal surfaces, and can also be described by
$\epsilon z=\operatorname{arccosh}\left(  \epsilon r\right)  .$

Let $k>0$ be a parameter. For each $a$ large, let
\[
\Omega_{a}:=\left\{  \left(  r,z\right)  :r\in\left[  0,a\right]  ,z\in\left[
0,b_{\varepsilon}\right]  \right\}  ,
\]
where%
\begin{equation}
b_{\varepsilon}=k\operatorname{arccosh}\left(  k^{-1}a\right)  +2+\delta
_{\varepsilon}. \label{b}%
\end{equation}
Set $L_{a}:=L_{1,a}\cup L_{2,a},$ where
\[
L_{1,a}:=\left\{  \left(  a,z\right)  :z\in\left[  0,b_{\varepsilon}\right]
\right\}  , \quad\mbox{and} \quad L_{2,a}:=\left\{  \left(  r,b_{\varepsilon
}\right)  :r\in\left[  0,a\right]  \right\}  .
\]
For fixed $k,$ we then define a function $\omega=\omega\left(  r,z\right)  ,$
depending on the parameter $\varepsilon$ and $a,$ to be
\[
\omega\left(  r,z\right)  =w_{\varepsilon,k\operatorname{arccosh}\left(
k^{-1}a\right)  -\varepsilon}\left(  z-k\operatorname{arccosh}\left(
k^{-1}a\right)  \right)  .
\]

Although eventually we are interested in solutions of the free boundary
problem in the whole space $\mathbb{R}^{3},$ it will be crucial to study
solutions $u=u\left(  r,z\right)  $ of the following regularized problem in
the bounded cylindrical domain $\Omega_{a},$ with mixed boundary condition:
\begin{equation}
\left\{
\begin{array}
[c]{l}%
-\partial_{r}^{2}u-\frac{1}{r}\partial_{r}u-\partial_{z}^{2}u+\frac{1}%
{2}F_{\varepsilon}^{\prime}\left(  u\right)  =0\text{ in }\Omega_{a},\\
\partial_{r}u\left(  0,z\right)  =0,\partial_{z}u\left(  r,0\right)  =0,\\
u=\omega\text{, on }L_{a}.
\end{array}
\right.  \label{w}%
\end{equation}

\subsection{Solutions of the regularized problems in $\Omega_{a}$ with
relatively small energy}

For each $a$ large, we would like to construct a mountain pass type solution
for the regularized problem $\left(  \ref{w}\right)  .$ We will first of all
look for two solutions $u_{1},u_{2}$ with relatively small energy. Minimaxing
in suitable class of paths of functions connecting $u_{1}$ and $u_{2}$, we
then obtain a mountain pass solution. Intuitively, $u_{1}$ will have nodal set
almost parallel to the horizontal $x_{1}$-$x_{2}$ plane, while the nodal set
of $u_{2}$ will be close to a vertical cylinder. Similar construction has been
carried out in \cite{Gui} for the Allen-Cahn equation in the two dimensional case.

\subsubsection{A solution with almost horizontal nodal set}

For fixed $\varepsilon,a,$ consider the following initial value problem for
the function $u=u\left(  t;r,z\right)  :$
\begin{equation}
\left\{
\begin{array}
[c]{l}%
\partial_{t}u=\Delta u-\frac{1}{2}F_{\varepsilon}^{\prime}\left(  u\right)
\text{ in }\Omega_{a}\times\left(  0,T\right)  ,\\
\partial_{r}u\left(  t;0,z\right)  =0,\partial_{z}u\left(  t;r,0\right)  =0,\\
u|_{L_{a}}=\omega,\\
u\left(  0;r,z\right)  =\omega\left(  r,z\right)  .
\end{array}
\right.  \label{Pa}%
\end{equation}
Since the constant function $\pm1$ solves the equation
\[
\partial_{t}u=\Delta u-\frac{1}{2}F_{\varepsilon}^{\prime}\left(
u_{\varepsilon}\right)  ,
\]
we infer that the solution $u$ of $\left(  \ref{Pa}\right)  $ satisfying
$-1<u<1.$ Hence the $L^{\infty}$ norm of the solution does not blow up in
finite time and by parabolic regularity, the solution can be extended to the
whole time interval $\left(  0,+\infty\right)  .$

Let us set
\[
E\left(  u\right)  :=\int_{\Omega_{a}}\left[  \left\vert \nabla u\right\vert
^{2}+F_{\varepsilon}\left(  u\right)  \right]  \geq0.
\]

\begin{lemma}
There exists a sequence $t_{n}\rightarrow+\infty,$ such that $u\left(
t_{n};\cdot\right)  $ converges to a solution $u_{1}$ of the problem
\begin{equation}
\left\{
\begin{array}
[c]{l}%
\Delta u-\frac{1}{2}F_{\varepsilon}^{\prime}\left(  u\right)  =0,\\
\partial_{r}u\left(  t;0,z\right)  =0,\partial_{z}u\left(  t;r,0\right)  =0,\\
u|_{L_{a}}=\omega.
\end{array}
\right.  \label{bou}%
\end{equation}

\end{lemma}

\begin{proof}
Direct computation yields%
\[
E^{\prime}\left(  u\left(  t\right)  \right)  =-2\int_{\Omega_{a}}\left\vert
\partial_{t}u\right\vert ^{2}\leq0.
\]
Hence $E\left(  u\left(  t\right)  \right)  $ is decreasing and uniformly
bounded. It also follows that
\[
\int_{0}^{+\infty}\int_{\Omega_{a}}\left\vert \partial_{t}u\right\vert
^{2}<+\infty.
\]
Hence there exists a sequence $t_{n}\rightarrow+\infty$ such that
\[
\int_{\Omega_{a}}\left\vert \partial_{t}u\left(  t_{n};\cdot\right)
\right\vert ^{2}\rightarrow0.
\]
We then get a P.S. sequence(in the natural functional space $H^{0,1}$, see
\cite{Gui} for related discussion) $\left\{  u\left(  t_{n};\cdot\right)
\right\}  $ for the functional $E$(i.e., $E\left(  u\left(  t_{n}%
;\cdot\right)  \right)  \leq C,$ $dE\left[  u\left(  t_{n};\cdot\right)
\right]  \rightarrow0$). Since $E$ satisfies the P.S. condition, using
standard variational arguments, we may extract a subsequence converging to a
solution $u_{1}$ of $\left(  \ref{bou}\right)  .$
\end{proof}

\begin{lemma}
$u_{1}$ is monotone in the following sense:
\[
\partial_{z}u_{1}>0\text{ and }\partial_{r}u_{1}<0,\text{ in }\Omega_{a}.
\]

\end{lemma}

\begin{proof}
The fact that $\partial_{z}u_{1}>0$ follows from the moving plane argument. It
remains to prove $\partial_{r}u_{1}<0.$

By Lemma \ref{sub}, we know that $\omega$ is a subsolution:%
\[
-\omega^{\prime\prime}+\frac{1}{2}F_{\varepsilon}^{\prime}\left(
\omega\right)  \leq0.
\]
In particular,
\[
\partial_{t}\omega-\Delta\omega+\frac{1}{2}F_{\varepsilon}^{\prime}\left(
\omega\right)  \leq0.
\]
Since $u\left(  0;r,z\right)  =\omega\left(  r,z\right)  ,$ parabolic
comparison principle (cf. \cite[Proposition 52.6]{Quit}) tells us that
$u\left(  t;\cdot\right)  \geq\omega\left(  \cdot\right)  $ in $\Omega_{a},$
for all $t\geq0.$ This then implies that $\partial_{r}u<0$ on $L_{1,a}$ for
any $t.$ Now the function $\phi:=\partial_{x}u$ satisfies
\[
\partial_{t}\phi-\Delta\phi+\frac{1}{2}F_{\varepsilon}^{\prime\prime}\left(
u\right)  \phi=0
\]
and $\phi\left(  0;\cdot\right)  =0$ and
\[
\phi\left(  t;\cdot\right)  \leq0\text{ on }\partial\left\{  \Omega_{a}%
\cap\left\{  x>0\right\}  \right\}  .
\]
Hence by the parabolic maximum principle (cf. \cite[Proposition 52.4]{Quit}),
$\phi\left(  t;\cdot\right)  \leq0$ on $\Omega_{a}.$ This proves the
monotonicity of $u_{1}$ in $r.$
\end{proof}

\subsubsection{A solution with almost vertical nodal set}

We shall construct a second solution $u_{2}$ whose nodal set is close to a
vertical cylinder. The energy of $u_{2}$ will be less than that of $u_{1}.$ To
show the existence of $u_{2},$ we still use the parabolic flow.

Let $U_{2}>u_{1}$ be a function such that $\partial_{r}U_{2}\leq0,$
$\partial_{z}U_{2}\geq0,$ and
\[
E\left(  U_{2}\right)  \leq10ka\ln a.
\]
Roughly speaking, we can construct $U_{2}$ whose nodal sets are almost
vertical, and locally in the direction transverse to the nodal set, it looks
like the one dimensional solution $H_{\varepsilon}.$ Now consider the solution
$u$ of the problem
\[
\left\{
\begin{array}
[c]{l}%
\partial_{t}u=\Delta u-\frac{1}{2}F_{\varepsilon}^{\prime}\left(  u\right)
,t\in\left(  0,+\infty\right)  ,\\
\partial_{r}u\left(  t;0,z\right)  =0,\partial_{z}u\left(  t;r,0\right)  =0,\\
u|_{L_{a}}=\omega,\\
u\left(  0;r,z\right)  =U_{2}.
\end{array}
\right.
\]
Similarly as before, we can show that there is a sequence $t_{n}%
\rightarrow+\infty,$ such that $u\left(  t_{n},\cdot\right)  $ converges to a
solution $u_{2}$ of $\left(  \ref{w}\right)  .$ Since $U_{2}>\omega,$ by the
comparison principle, we have $u_{2}>u_{1}.$ We also have
\[
\partial_{z}u_{2}>0\text{ and }\partial_{r}u_{2}<0\text{ in }\Omega_{a}.
\]

\subsection{Mountain pass type solutions}

We have so far obtained two solutions $u_{1}$, $u_{2},$ with $u_{1}<u_{2}.$
Now we would like to construct a mountain pass type solution using $u_{1}$ and
$u_{2}.$ Let $\mathcal{E}$ be the set of $C^{1}$ functions $\phi$ satisfying
the following properties:

\noindent(I) $u_{1}<\phi<u_{2}$ in $\Omega_{a}$,

\noindent(II) $\partial_{z}\phi>0;\partial_{r}\phi<0,$ in $\Omega_{a},$

\noindent(III) $\phi|_{L_{a}}=\omega,$

\noindent(IV)$\partial_{r}\phi\left(  0,z\right)  =0,\partial_{z}\phi\left(
r,0\right)  =0.$

To proceed, we define
\[
e_{\varepsilon}=\int_{\mathbb{R}}\left[  H_{\varepsilon}^{\prime
2}+F_{\varepsilon}\left(  H_{\varepsilon}\right)  \right]  =2\int_{-1}%
^{1}\sqrt{F_{\varepsilon}\left(  s\right)  }ds.
\]
Note that $e_{\varepsilon}\rightarrow4,$ as $\varepsilon\rightarrow0.$ Let
$\varepsilon_{0}$ be a fixed small positive constant. For each $\varepsilon
\in\left(  0,\varepsilon_{0}\right)  ,$ we can construct a family of $C^{1}$
functions $\eta_{\varepsilon}^{\ast}\left(  s;r,z\right)  $ depending
continuously on $s,$ such that $\eta_{\varepsilon}^{\ast}\left(
s;\cdot\right)  \in\mathcal{E}$ for any $s\in\left[  0,1\right]  $ and
\[
\eta_{\varepsilon}^{\ast}\left(  0;\cdot\right)  =u_{1},\eta_{\varepsilon
}^{\ast}\left(  1;\cdot\right)  =u_{2}.
\]
Moreover, we require $\partial_{s}\eta_{\varepsilon}^{\ast}\left(
s;r,z\right)  \geq0,$ and
\begin{equation}
\max_{s\in\left[  0,1\right]  }E\left(  \eta_{\varepsilon}^{\ast}\left(
s\right)  \right)  \leq\frac{a^{2}e_{\varepsilon}}{2}+\frac{e_{\varepsilon}%
}{2}k^{2}\ln a+C. \label{estar}%
\end{equation}
We may also assume that $\left\vert \nabla_{\left(  r,z\right)  }%
\eta_{\varepsilon}^{\ast}\left(  s\right)  \right\vert $ is uniformly bounded
for $s\in\left[  0,1\right]  $ and $\varepsilon\in\left(  0,\varepsilon
_{0}\right)  .$ The existence of this family of solutions essentially follow
from geometric properties of catenoids.

Now we shall consider the solution $u=u^{\ast}\left(  t;s;r,z\right)  $ of the
initial value problem%
\[
\left\{
\begin{array}
[c]{l}%
\partial_{t}u^{\ast}=\Delta u^{\ast}-\frac{1}{2}F_{\varepsilon}^{\prime
}\left(  u^{\ast}\right)  ,t\in\left(  0,+\infty\right)  ,\\
\partial_{r}u^{\ast}\left(  t;s;0,z\right)  =0,\partial_{z}u^{\ast}\left(
t;s;r,0\right)  =0,\\
u^{\ast}\left(  t;s;\cdot\right)  |_{L_{a}}=\omega,\\
u^{\ast}\left(  0;s;\cdot\right)  =\eta_{\varepsilon}^{\ast}\left(
s;\cdot\right)  .
\end{array}
\right.
\]
Using the order preserving property of the parabolic flow, we know that for
each $t\geq0$ and $s\in\left[  0,1\right]  ,$ $u^{\ast}\left(  t;s;\cdot
\right)  \in\mathcal{E}$. Moreover, $\partial_{s}u^{\ast}\left(
t;s;\cdot\right)  \geq0.$

Let
\[
P=\left\{  u^{\ast}\left(  t;\cdot\right)  :t\in\lbrack0,+\infty)\right\}  .
\]
We define
\[
c^{\ast}=\min_{\eta\in P}\max_{s\in\left[  0,1\right]  }E\left(  \eta\left(
s;\cdot\right)  \right)  .
\]
The following lemma gives us the upper bound on $c^{\ast}.$

\begin{lemma}
\label{cupper}There exists a constant $C$ independent of $a$ and
$\varepsilon,$ such that%
\[
c^{\ast}\leq\frac{a^{2}e_{\varepsilon}}{2}+\frac{e_{\varepsilon}}{2}k^{2}\ln
a+C.
\]

\end{lemma}

\begin{proof}
This follows directly from the property $\left(  \ref{estar}\right)  $ of
$\eta_{\varepsilon}^{\ast}$ and the fact that the energy $E$ is decreasing
along the parabolic flow.
\end{proof}

To prove the existence of mountain pass solution, we need to get a lower bound
for $c^{\ast}.$ It turn out that the estimate of the lower bound is much more delicate.

\begin{lemma}
\label{e1}Suppose $r_{0}\in\left[  k,a\right]  .$ Let $\xi$ be a $C^{1}$
function defined on $\left[  r_{0},a\right]  $ such that $\xi\left(
r_{0}\right)  =0$ and $\xi\left(  a\right)  =k\operatorname{arccosh}\left(
k^{-1}a\right)  .$ Then
\[
\int_{r_{0}}^{a}\sqrt{1+\xi^{\prime2}\left(  r\right)  }rdr\geq\frac{1}%
{2}a^{2}-\frac{1}{2}r_{0}^{2}+\frac{k^{2}}{2}\ln a-C_{k},
\]
where $C_{k}$ is independent of $r_{0}$ and $a.$
\end{lemma}

\begin{proof}
Define a new function
\[
\bar{\xi}\left(  r\right)  :=\left\{
\begin{array}
[c]{l}%
\xi\left(  r\right)  ,r\in\left[  r_{0},a\right]  ,\\
0,r\in\left[  k,r_{0}\right]  .
\end{array}
\right.
\]
Then using the fact that the function $g\left(  r\right)
:=k\operatorname{arccosh}\left(  k^{-1}r\right)  $ represents a minimal
surface (a catenoid) and hence it has minimizing area, we get
\begin{align*}
\int_{k}^{a}\sqrt{1+\bar{\xi}^{\prime2}\left(  r\right)  }rdr  &  \geq\int%
_{k}^{a}\sqrt{1+g^{\prime2}\left(  r\right)  }rdr\\
&  =\int_{0}^{k\operatorname{arccosh}\left(  k^{-1}a\right)  }\sqrt
{1+\sinh^{2}\left(  k^{-1}z\right)  }k\cosh\left(  k^{-1}z\right)  dz\\
&  =\frac{k^{2}}{2}\operatorname{arccosh}\left(  k^{-1}a\right)  +\frac{a^{2}%
}{2}\sqrt{1-k^{2}a^{-2}}.
\end{align*}
Since $\int_{k}^{r_{0}}\sqrt{1+\bar{\xi}^{\prime2}\left(  r\right)  }%
rdr=\frac{1}{2}r_{0}^{2}-\frac{1}{2}k^{2},$ we then get
\[
\int_{r_{0}}^{a}\sqrt{1+\bar{\xi}^{\prime2}\left(  r\right)  }rdr\geq
\frac{a^{2}}{2}-\frac{r_{0}^{2}}{2}+\frac{k^{2}}{2}\ln a-C_{k}.
\]
This is the desired estimate.
\end{proof}

\begin{proposition}
\label{moun}For $\varepsilon$ small enough, there exists a constant $C$
independent of $a,\varepsilon,$ such that%
\[
c^{\ast}\geq\frac{1}{2}a^{2}e_{\varepsilon}+\frac{k^{2}}{20}\ln a-C.
\]

\end{proposition}

\begin{proof}
Let $\eta\in P.$ Since $\eta$ is a continuous family of $C^{1}$ functions from
$u_{1}$ to $u_{2}$, we know that there is a $s_{0}\in\left(  0,1\right)  $,
such that the function $u\left(  \cdot\right)  :=\eta\left(  s_{0}%
;\cdot\right)  $ is equal to $0$ at the point $\left(  k,\frac{k}{10}\ln
a\right)  .$ We introduce the notation%
\[
\Omega_{a}^{-}=\left\{  X\in\Omega_{a}:u\left(  X\right)  <0\right\}  ,
\]
\[
\Omega_{a}^{+}=\left\{  X\in\Omega_{a}:u\left(  X\right)  >0\right\}  .
\]

By the coarea formula, we have%
\begin{align*}
\int_{\Omega_{a}^{+}}\left[  \left\vert \nabla u\right\vert ^{2}%
+F_{\varepsilon}\left(  u\right)  \right]   &  \geq2\int_{\Omega_{a}^{+}%
}\left[  \left\vert \nabla u\right\vert \sqrt{F_{\varepsilon}\left(  u\right)
}\right] \\
&  =2\int_{0}^{1}A\left(  s\right)  \sqrt{F_{\varepsilon}\left(  s\right)
}ds,
\end{align*}
where
\[
A\left(  s\right)  =\text{Area of }\left\{  X:u\left(  X\right)  =s\right\}
.
\]
Since $u$ is monotone in $r$ and $z,$ we deduce that for $s\in\left(
0,1\right)  ,$
\[
A\left(  s\right)  \geq\frac{1}{2}a^{2}-C,
\]
where $C$ does not depend on $a$ and $\varepsilon.$ Hence
\begin{equation}
\int_{\Omega_{a}^{+}}\left[  \left\vert \nabla u\right\vert ^{2}%
+F_{\varepsilon}\left(  u\right)  \right]  \geq\frac{1}{4}e_{\varepsilon}%
a^{2}-C. \label{+}%
\end{equation}

Next we estimate the energy in the region $\Omega_{a}^{-},$ which is more
involved. For $r\geq0,$ we define $s=u\left(  r,0\right)  .$ It is a function
of $r.$ Applying Lemma \ref{e1}, we infer that for $s\leq\min\left\{
0,u\left(  0,0\right)  \right\}  ,$
\[
A\left(  s\right)  \geq\frac{1}{2}a^{2}-\frac{1}{2}r^{2}+\frac{k^{2}}{2}\ln
a-C.
\]
Using this estimate, we find that
\begin{align}
&  \int_{\Omega_{a}^{-}}\left[  \left\vert \nabla u\right\vert ^{2}%
+F_{\varepsilon}\left(  u\right)  \right]  \geq\int_{-1}^{0}A\left(  s\right)
\sqrt{F_{\varepsilon}\left(  s\right)  }ds\nonumber\\
&  \geq\int_{\min\left\{  0,u\left(  0,0\right)  \right\}  }^{0}A\left(
s\right)  \sqrt{F_{\varepsilon}\left(  s\right)  }ds\nonumber\\
&  +\left(  \frac{1}{2}a^{2}+\frac{k^{2}}{2}\ln a\right)  \int_{-1}%
^{\min\left\{  0,u\left(  0,0\right)  \right\}  }\sqrt{F_{\varepsilon}\left(
s\right)  }ds\nonumber\\
&  -\frac{1}{2}\int_{-1}^{\min\left\{  0,u\left(  0,0\right)  \right\}  }%
r^{2}\sqrt{F_{\varepsilon}\left(  s\right)  }ds-C. \label{e0}%
\end{align}
We would like to estimate the last integral. For this purpose, define a new
function $\phi\left(  r\right)  :=F_{\varepsilon}\left(  u\left(  r,0\right)
\right)  =F_{\varepsilon}\left(  s\right)  .$ We distinguish two possibilities.

Case 1.
\begin{equation}
\int_{0}^{a}\phi\left(  r\right)  rdr>\frac{k^{2}}{10}\ln a. \label{low}%
\end{equation}
In this case, we have
\begin{align*}
E\left(  u\right)   &  =\int_{\Omega_{a}\cap\left\{  z>1\right\}  }\left[
\left\vert \nabla u\right\vert ^{2}+F_{\varepsilon}\left(  u\right)  \right]
\\
&  +\int_{\Omega_{a}\cap\left\{  0<z<1\right\}  }\left[  \left\vert \nabla
u\right\vert ^{2}+F_{\varepsilon}\left(  u\right)  \right] \\
&  \geq\frac{a^{2}}{2}e_{\varepsilon}+\int_{\Omega_{a}\cap\left\{
0<z<1\right\}  }F_{\varepsilon}\left(  u\right)  -C.
\end{align*}
Due to the monotonicity of $u$ in the $r$ and $z$ direction, we have%
\begin{align*}
\int_{\Omega_{a}\cap\left\{  0<z<1\right\}  }F_{\varepsilon}\left(  u\right)
&  \geq\int_{\Omega_{a}\cap\left\{  0<z<1\right\}  }F_{\varepsilon}\left(
u\left(  r,0\right)  \right) \\
&  =\int_{0}^{a}\phi\left(  r\right)  rdr.
\end{align*}
It then follows from $\left(  \ref{low}\right)  $ that
\[
E\left(  u\right)  \geq\frac{a^{2}}{2}e_{\varepsilon}+\frac{k^{2}}{10}\ln
a-C.
\]
This is the desired estimate.

Case 2.
\begin{equation}
\int_{0}^{a}\phi\left(  r\right)  rdr\leq\frac{k^{2}}{10}\ln a. \label{low2}%
\end{equation}
In this case, we write
\[
\int_{-1}^{\min\left\{  0,u\left(  0,0\right)  \right\}  }r^{2}\sqrt
{F_{\varepsilon}\left(  s\right)  }ds=\int_{-1}^{-1+\frac{\varepsilon}{2}%
}r^{2}\sqrt{F_{\varepsilon}\left(  s\right)  }ds+\int_{-1+\frac{\varepsilon
}{2}}^{\min\left\{  0,u\left(  0,0\right)  \right\}  }r^{2}\sqrt
{F_{\varepsilon}\left(  s\right)  }ds.
\]
Let us estimate these two integrals separately.

Recall that when $s\in\left[  -1,-1+\frac{\varepsilon}{2}\right]  ,$
$\phi\left(  r\right)  =F_{\varepsilon}\left(  s\right)  =\varepsilon
^{-2}\left(  s+1\right)  ^{2}.$ Let $\bar{t}$ be the point where $u\left(
\bar{t},0\right)  =-1+\frac{\varepsilon}{2}.$ Then%
\begin{equation}
\int_{-1+\frac{\varepsilon}{2}}^{\min\left\{  0,u\left(  0,0\right)  \right\}
}r^{2}\sqrt{F_{\varepsilon}\left(  s\right)  }ds\leq\bar{t}^{2}. \label{F}%
\end{equation}
On the other hand, using the monotonicity of $\phi$ and $\left(
\ref{low2}\right)  ,$ we get
\begin{equation}
\phi\left(  r\right)  r^{2}\leq2\int_{0}^{r}\phi\left(  t\right)  tdt\leq
\frac{k^{2}}{5}\ln a,\text{ for any }t\in\left[  0,a\right]  . \label{fit}%
\end{equation}
This together with $\phi\left(  \bar{t}\right)  =\frac{1}{2}$ tells us that
$\bar{t}^{2}\leq\frac{2k^{2}}{5}\ln a.$ Hence in view of $\left(
\ref{F}\right)  ,$ we find that
\begin{equation}
\int_{-1+\frac{\varepsilon}{2}}^{\min\left\{  0,u\left(  0,0\right)  \right\}
}r^{2}\sqrt{F_{\varepsilon}\left(  s\right)  }ds\leq\bar{t}^{2}\leq
\frac{2k^{2}}{5}\ln a. \label{e4}%
\end{equation}

Next, we compute
\begin{align*}
\int_{-1}^{-1+\frac{\varepsilon}{2}}r^{2}\sqrt{F_{\varepsilon}\left(
s\right)  }ds  &  =-\frac{\varepsilon}{2}\int_{\bar{t}}^{a}r^{2}\phi^{\prime
}\left(  r\right)  dr\\
&  =-\frac{\varepsilon}{2}\left(  \phi\left(  a\right)  a^{2}-\phi\left(
\bar{t}\right)  \bar{t}^{2}\right)  +\varepsilon\int_{\bar{t}}^{a}\phi\left(
t\right)  tdt.
\end{align*}
Applying $\left(  \ref{low2}\right)  $ and $\left(  \ref{fit}\right)  ,$ we
get
\begin{equation}
\int_{-1}^{-1+\frac{\varepsilon}{2}}r^{2}\sqrt{F_{\varepsilon}\left(
s\right)  }ds\leq\frac{2k^{2}\varepsilon}{5}\ln a. \label{e5}%
\end{equation}
Combining $\left(  \ref{+}\right)  ,\left(  \ref{e0}\right)  ,\left(
\ref{e4}\right)  ,\left(  \ref{e5}\right)  ,$ we obtain
\begin{align*}
E\left(  u\right)   &  \geq\int_{\Omega_{a}^{-}}\left[  \left\vert \nabla
u\right\vert ^{2}+F_{\varepsilon}\left(  u\right)  \right]  +\int_{\Omega
_{a}^{+}}\left[  \left\vert \nabla u\right\vert ^{2}+F_{\varepsilon}\left(
u\right)  \right] \\
&  \geq\frac{1}{2}a^{2}e_{\varepsilon}+\frac{k^{2}}{2}e_{\varepsilon}\ln
a-\frac{1}{2}\int_{-1}^{\min\left\{  0,u\left(  0,0\right)  \right\}  }%
r^{2}\sqrt{F_{\varepsilon}\left(  s\right)  }ds-C\\
&  \geq\frac{1}{2}a^{2}e_{\varepsilon}+\frac{k^{2}}{2}e_{\varepsilon}\ln
a-\frac{k^{2}}{5}\ln a-\frac{k^{2}\varepsilon}{5}\ln a-C\\
&  \geq\frac{1}{2}a^{2}e_{\varepsilon}+\frac{k^{2}}{20}\ln a-C,
\end{align*}
provided that $\varepsilon$ is small enough. This finishes the proof.
\end{proof}

\begin{remark}
For the purpose of obtaining a mountain pass solution, one only need to prove
the estimate
\[
c^{\ast}\geq\frac{1}{2}a^{2}e_{\varepsilon}+\delta,
\]
for some universal constant $\delta.$ See Lemma \ref{a} for the corresponding
estimate in the higher dimensional case.
\end{remark}

\begin{proposition}
Let $a$ be large enough. Then there exists a mountain pass solution
$U_{\varepsilon}=U_{\varepsilon,a}$ to $\left(  \ref{w}\right)  .$ Moreover,
$\partial_{r}U_{\varepsilon}>0,\partial_{z}U_{\varepsilon}<0$ in $\Omega_{a}.$
\end{proposition}

\begin{proof}
By Proposition \ref{moun},
\begin{equation}
c^{\ast}\geq\frac{1}{2}a^{2}e_{\varepsilon}+\frac{k^{2}}{20}\ln a-C>\max
_{i=1,2}E\left(  u_{i}\right)  , \label{energy}%
\end{equation}
provided that $a$ is sufficiently large. Standard arguments in variational
methods yield the existence of a solution $U_{\varepsilon,a}$ whose energy is
equal to $c^{\ast}.$
\end{proof}

\section{Asymptotic analysis of $\left\{  U_{\varepsilon}\right\}  $ and
regularity of the free boundary of the limiting solution\label{Sec3}}

For each fixed large constant $a,$ we have obtained a family of solutions
$U_{\varepsilon}$ to the regularized problem. Using arguments of Section 1.2
of Caffarelli-Salsa\cite{Cs}, we can show that $\left\vert \nabla
U_{\varepsilon,a}\right\vert \leq C.$ Therefore, $U_{\varepsilon,a}$ converges
in $C^{0,\alpha}\left(  \Omega_{a}\right)  $ to a function $V_{a}.$ Since
$F_{\varepsilon}$ converges on any compact subinterval of $\left(
-1,1\right)  $ to $1,$ $V_{a}$ is a harmonic function in the region $\Xi
_{a}:=\left\{  \left\vert V_{a}\right\vert <1\right\}  \cap\Omega_{a}.$ Recall
that $U_{\varepsilon,a}$ is monotone, hence
\[
\partial_{r}V_{a}<0\text{ and }\partial_{z}V_{a}>0\text{ in }\Xi_{a}.
\]

In this section, we show that $V_{a}$ satisfies the free boundary condition
$\left\vert \nabla V_{a}\right\vert =1$ on $\partial\left\{  \left\vert
V_{a}\right\vert <1\right\}  \cap\Omega_{a}$ in the classical sense. Let us
introduce the notation
\[
\digamma_{a}:=\partial\Xi_{a}\cap\Omega_{a}.
\]
We also define
\begin{equation}
\digamma_{a}^{+}=\digamma_{a}\cap\left\{  V_{a}=1\right\}  ,\text{ }%
\digamma_{a}^{-}=\digamma_{a}\cap\left\{  V_{a}=-1\right\}  . \label{CF}%
\end{equation}

To investigate the regularity property of the free boundary $\digamma_{a},$
the first step is to show that the free boundary is nondegenerated in the
sense of $\cite{Alt}.$ We use $B_{\rho}\left(  X\right)  $ to denote the ball
of radius $\rho$ with center $X$ in $\mathbb{R}^{3}.$

\begin{lemma}
Let $x_{0}=\left(  r_{0},z_{0}\right)  \in$ $\digamma_{a}^{+}$ with $z_{0}>0.$
Let $\rho<\frac{1}{2}.$ For any ball $B_{\rho}\subset B_{\frac{z_{0}}{2}%
}\left(  x_{0}\right)  ,$
%if $V_{a}$ is not identically zero in $B_{\rho},$
then
\[
\rho^{-3}\int_{\partial B_{\rho}}V_{a}\geq C>0.
\]

\end{lemma}

\begin{proof}
Checking the details of the proof of Lemma 3.4 in \cite{Alt}, we find that to
prove this nondegeneracy property, we need to show the local minimizing
property of $V_{a}$, i.e. compare the energy of $V_{a}$ with another carefully
chosen test function larger than $V_{a}$. To do this, we shall use suitable
minimizing property of the function $U_{\varepsilon}$ and sending
$\varepsilon$ to $0.$

Let $B_{\rho}$ be a ball of radius $\rho$ in $B_{\frac{z_{0}}{2}}\left(
x_{0}\right)  .$ For each fixed small $\varepsilon,$ consider the smooth
family of functions $U_{\varepsilon}\left(  r,z-k\right)  $, with
\[
0\leq k<b_{\varepsilon}-z_{0}-\rho,
\]
where $b_{\varepsilon}$ is the constant appeared in $\left(  \ref{b}\right)
.$ Since $U_{\varepsilon}$ is monotone in $z,$ we have
\[
U_{\varepsilon}\left(  r,z-k_{1}\right)  <U_{\varepsilon}\left(
r,z-k_{2}\right)  ,\text{ if \ }k_{1}<k_{2}.
\]
Using this monotone family of functions, we can construct a calibration, using
the theory developed in $\cite{Cabre}.$ The arguments of Theorem 4.5 in
$\cite{Cabre}$ then tell us that
\begin{equation}
\int_{B_{\rho}}\left[  \left\vert \nabla U_{\varepsilon}\right\vert
^{2}+F_{\varepsilon}\left(  U_{\varepsilon}\right)  \right]  \leq\int%
_{B_{\rho}}\left[  \left\vert \nabla\eta\right\vert ^{2}+F_{\varepsilon
}\left(  \eta\right)  \right]  , \label{ener}%
\end{equation}
for any smooth function $\eta$ satisfying $\eta=U_{\varepsilon}$ on $\partial
B_{\rho},$ and
\[
U_{\varepsilon}\leq\eta\leq U_{\varepsilon}\left(  r,z-b_{\varepsilon}%
+z_{0}+\rho\right)  .
\]
We observe that due to monotonicity,
\begin{equation}
U_{\varepsilon}\left(  r,z-b_{\varepsilon}+z_{0}+\rho\right)  \geq
1-\varepsilon^{2}. \label{close}%
\end{equation}

Following Alt-Caffarelli $\left(  \cite{Alt}\right)  ,$ we define%
\[
g_{\beta}\left(  X\right)  =\beta\left(  \ln\left\vert X\right\vert -\ln
\beta\right)  ,
\]%
\[
w_{\varepsilon}\left(  X\right)  =\min\left\{  c_{0}g_{\frac{\rho}{4}}\left(
X-x_{0}\right)  ,1-\varepsilon^{2}\right\}  ,
\]
and let $W_{\varepsilon}=\max\left\{  U_{\varepsilon},w_{\varepsilon}\right\}
.$ Here $c_{0}$ is the maximum constant choose such that $w_{\varepsilon}\leq
U_{\varepsilon}$ on $\partial B_{\rho}.$

Since $U_{\varepsilon}\leq W_{\varepsilon}$ and $U_{\varepsilon}%
=W_{\varepsilon}$ on $\partial B_{\rho},$ by $\left(  \ref{ener}\right)  ,$
\[
\int_{B_{\rho}}\left[  \left\vert \nabla U_{\varepsilon}\right\vert
^{2}+F_{\varepsilon}\left(  U_{\varepsilon}\right)  \right]  \leq\int%
_{B_{\rho}}\left[  \left\vert \nabla W_{\varepsilon}\right\vert ^{2}%
+F_{\varepsilon}\left(  W_{\varepsilon}\right)  \right]  .
\]
On the other hand, for any subdomain $\Omega\subset B_{\rho},$%
\begin{equation}
\int_{\Omega}\left[  \left\vert \nabla V_{a}\right\vert ^{2}+\chi_{\left(
-1,1\right)  }\left(  V_{a}\right)  \right]  \leq\lim\inf_{\varepsilon
\rightarrow0}\int_{\Omega}\left[  \left\vert \nabla U_{\varepsilon}\right\vert
^{2}+F_{\varepsilon}\left(  U_{\varepsilon}\right)  \right]  . \label{inf}%
\end{equation}
Letting $\varepsilon\rightarrow0,$ using $\left(  \ref{inf}\right)  $ in the
region where $w_{\varepsilon}\geq U_{\varepsilon},$ we obtain
\begin{equation}
\int_{B_{\rho}}\left[  \left\vert \nabla V_{a}\right\vert ^{2}+\chi_{\left(
-1,1\right)  }\left(  V_{a}\right)  \right]  \leq\int_{B_{\rho}}\left[
\left\vert \nabla W\right\vert ^{2}+\chi_{\left(  -1,1\right)  }\left(
W\right)  \right]  . \label{es}%
\end{equation}
Once $\left(  \ref{es}\right)  $ is proved, we may proceed as Lemma 3.4 of
\cite{Alt} to conclude the proof.
\end{proof}

Next we study the nondegeneracy around $\digamma_{a}^{-}.$

\begin{lemma}
\label{plus}Let $x_{0}=\left(  r_{0},z_{0}\right)  \in\digamma_{a}^{-}.$
Suppose there exists $\delta>0$ such that%
\[
\digamma_{a}^{-}\cap\left\{  \left(  r,z\right)  :r\in\left[  r_{0}%
-\delta,r_{0}+\delta\right]  \right\}  \subset\left\{  \left(  r,z\right)
:z>2\delta\right\}  .
\]
Then for any ball $B_{\rho}\subset B_{\delta}\left(  x_{0}\right)  ,$ if
$V_{a}$ is not identically zero in $B_{\rho},$ we have
\[
\rho^{-3}\int_{\partial B_{\rho}}V_{a}\geq C>0.
\]

\end{lemma}

\begin{proof}
Let $B_{\rho}$ be the ball of radius $\rho$ in $B_{\delta}\left(
x_{0}\right)  $ with center $\left(  r_{\ast},z_{\ast}\right)  .$ Consider the
family of functions $U_{\varepsilon}\left(  r,z-k\right)  ,$ with $-\left(
z_{\ast}-\rho\right)  <k\leq0.$ Due to monotonicity,
\[
U_{\varepsilon}\left(  r,z-k_{1}\right)  \leq U_{\varepsilon}\left(
r,z-k_{2}\right)  ,\text{ if }k_{1}<k_{2}.
\]
The same arguments as Lemma \ref{plus} yield
\[
\int_{B_{\rho}}\left[  \left\vert \nabla U_{\varepsilon}\right\vert
^{2}+F_{\varepsilon}\left(  U_{\varepsilon}\right)  \right]  \leq\int%
_{B_{\rho}}\left[  \left\vert \nabla\eta\right\vert ^{2}+F_{\varepsilon
}\left(  \eta\right)  \right]  ,
\]
for any smooth function $\eta$ satisfying $\eta=U_{\varepsilon}$ on $\partial
B_{\rho},$ and
\[
U_{\varepsilon}\left(  r,z+z_{\ast}-\rho\right)  \leq\eta\leq U_{\varepsilon
}\text{ in }B_{\rho}.
\]
While in Lemma \ref{plus} we know from $\left(  \ref{close}\right)  $ that the
function $U_{\varepsilon}\left(  r,z-b_{\varepsilon}+z_{0}+\rho\right)  $ is
close enough to $1,$ we do not have similar estimate for $U_{\varepsilon
}\left(  r,z+z_{\ast}-\rho\right)  $ up to now. Nevertheless, we would like to
show
\begin{equation}
F_{\varepsilon}\left[  U_{\varepsilon}\left(  r,z+z_{\ast}-\rho\right)
\right]  \rightarrow0\text{ in }B_{\rho},\text{ as }\varepsilon\rightarrow0.
\label{0}%
\end{equation}
Once this is proved, the rest of the proof is same as Lemma \ref{plus}.

For each $r\in\left[  r_{0}-\delta,r_{0}+\delta\right]  ,$ we define%
\[
d\left(  r\right)  =\inf\left\{  z:\left(  r,z\right)  \in\digamma
^{-}\right\}  ,
\]
and
\[
\Lambda=\left\{  \left(  r,z\right)  :r\in\left[  r_{0}-\delta,r_{0}%
+\delta\right]  ,z<d\left(  r\right)  \right\}  .
\]
The measure of a set $S$ will be denoted by $\left\vert S\right\vert .$

We claim that for each fixed constant $K>0$%
\begin{equation}
\lim_{\varepsilon\rightarrow0}\left\vert \Lambda\cap\left\{  \left\vert
F_{\varepsilon}^{\prime}\left(  U_{\varepsilon}\right)  \right\vert
>K\right\}  \right\vert =0. \label{K}%
\end{equation}
Suppose this were not true. Then we could find a subsequence $\left\{
\varepsilon_{n}\right\}  $ tending to $0,$ and $r_{1},r_{2},z_{1},z_{2},$
depending on $\varepsilon_{n},$ such that
\begin{equation}
F_{\varepsilon_{n}}^{\prime}\left(  U\left(  r,z\right)  \right)  >K,\text{
for }\left(  r,z\right)  \in D:=\left(  r_{1},r_{2}\right)  \times\left(
z_{1},z_{2}\right)  \subset\Lambda. \label{case1}%
\end{equation}
Moreover, we could assume $\left\vert z_{2}-z_{1}\right\vert =\left\vert
r_{2}-r_{1}\right\vert =\delta>0,$ where $\delta$ is independent of
$\varepsilon.$ Then in the region $D,$ $\Delta U_{\varepsilon_{n}%
}=F_{\varepsilon_{n}}^{\prime}\left(  U_{\varepsilon_{n}}\right)  \geq K.$ Let
$\phi$ be a function satisfying
\[
\Delta\phi=K\text{ in }D,\text{ }\phi=U_{\varepsilon_{n}}\text{ on }\partial
D.
\]
Then
\[
-\Delta\left(  U_{\varepsilon_{n}}-\phi\right)  \leq0\text{ in }D,\text{
}U_{\varepsilon_{n}}-\phi\leq0\text{ on }\partial D.
\]
Hence $U_{\varepsilon_{n}}\leq\phi$ in $D.$ In view of the fact that
$U_{\varepsilon_{n}}\rightarrow1$ in $D$ as $\varepsilon_{n}\rightarrow0,$ we
get $\phi<-1$ at the center of $D.$ This contradicts with the fact that
$U_{\varepsilon}\geq-1.$

To prove $\left(  \ref{0}\right)  ,$ we first show that for each fixed
$\left(  r^{\ast},z^{\ast}\right)  \in\Lambda$,
\begin{equation}
\lim_{\varepsilon\rightarrow0}F_{\varepsilon}\left(  U_{\varepsilon}\left(
r^{\ast},z^{\ast}\right)  \right)  =0. \label{zero}%
\end{equation}
Assume to the contrary that
\[
\lim\sup_{\varepsilon\rightarrow0}F_{\varepsilon}\left(  U_{\varepsilon
}\left(  r^{\ast},z^{\ast}\right)  \right)  =\xi>0.
\]
Then using $\left(  \ref{K}\right)  ,$ we could infer that in the region
$\Lambda^{\ast}:=\left\{  \left(  r,z\right)  :\left(  r,z\right)  \in
\Lambda,r<r^{\ast},z>z^{\ast}\right\}  ,$ $\Delta U_{\varepsilon}$ converges
pointwise to $0.$ Hence
\[
\Delta V_{a}=0\text{ in }\Lambda^{\ast}.
\]
This contradicts with the maximum principle. Hence we get $\left(
\ref{0}\right)  .$

We remark that once $\left(  \ref{zero}\right)  $ is proven, we can show
exponentially decay to $-1$ in $\Lambda,$ away from the free boundary points.
\end{proof}

Having obtained sufficiently fast decay to $\pm1$ away from the free boundary,
we prove that $V_{a}$ is a variational solution (see \cite{Weiss3} on a
discussion on this topic).

\begin{lemma}
$V_{a}$ is a variational solution in the sense that%
\begin{equation}
\int_{\Omega}\left\{  \left(  \left\vert \nabla V_{a}\right\vert ^{2}%
+\chi_{\left(  -1,1\right)  }\left(  V_{a}\right)  \right)  \operatorname{div}%
\phi-2\nabla V_{a}D\phi\left(  \nabla V_{a}\right)  ^{T}\right\}  =0.
\label{va}%
\end{equation}
for any $\phi\in C_{c}^{\infty}\left(  \Omega,\mathbb{R}^{3}\right)  .$
\end{lemma}

\begin{proof}
Since $U_{\varepsilon}$ is $C^{1}$ and
\[
\operatorname{div}\left[  \left(  \left\vert \nabla U_{\varepsilon}\right\vert
^{2}+F_{\varepsilon}\left(  U_{\varepsilon}\right)  \right)  \phi\right]
=\left(  \left\vert \nabla U_{\varepsilon}\right\vert ^{2}+F_{\varepsilon
}\left(  U_{\varepsilon}\right)  \right)  \operatorname{div}\phi-2\nabla
U_{\varepsilon}D\phi\left(  \nabla U_{\varepsilon}\right)  ^{T},
\]
we have%
\[
\int_{\Omega_{a}}\left\{  \left(  \left\vert \nabla U_{\varepsilon}\right\vert
^{2}+F_{\varepsilon}\left(  U_{\varepsilon}\right)  \right)
\operatorname{div}\phi-2\nabla U_{\varepsilon}D\phi\left(  \nabla
U_{\varepsilon}\right)  ^{T}\right\}  =0.
\]
Letting $\varepsilon\rightarrow0,$ using the fact that $\nabla U_{\varepsilon
}$ is uniformly bounded with respect to $\varepsilon$ and the exponential
decay to $0$ away from the free boundary, we get the desired result.
\end{proof}

\begin{lemma}
$\digamma_{a}^{\pm}$ is a smooth curve away from the origin, and $V_{a}$ is a
solution of the free boundary problem%
\[
\left\{
\begin{array}
[c]{l}%
\Delta V_{a}=0\text{ in }\Xi_{a},\\
\left\vert \nabla V_{a}\right\vert =1\text{ on }\digamma_{a}^{\pm}.
\end{array}
\right.
\]
Moreover, the energy of $V_{a}$ has the following lower bound estimate:
\begin{equation}
J\left(  V_{a}\right)  =\int_{\Omega_{a}}\left[  \left\vert \nabla
V_{a}\right\vert ^{2}+\chi_{\left(  -1,1\right)  }\left(  V_{a}\right)
\right]  \geq2a^{2}+\frac{k^{2}}{20}\ln a-C. \label{V}%
\end{equation}

\end{lemma}

\begin{proof}
Let $X\in\digamma_{a}^{\pm}$. Suppose first of all that $X$ is not on the $z$
axis. Since $V_{a}$ is nondegenerated and a variational solution, the Weiss
monotonicity formula(\cite{W,Weiss2}) and a standard blow up analysis tell us
that the blow up limit around $X$ is a cone. Due to rotational symmetry around
the $z$ axis, this is a two dimensional cone. Hence it must be trivial. Then
the usual regularity theory(\cite{Ca1,Ca2,Ca3}) of free boundary tells us that
around $X$ the free boundary is analytic. Now suppose $X$ is on the $z$ axis
and is not the origin. The blow up limit around $X$ will be the cone $\left(
\ref{cone}\right)  ,$ this contradicts with the monotonicity of $V_{a}$ in the
$z$ direction.

In view of the exponential decay of $\nabla U_{\varepsilon,a}$ to $0$ in
$\Omega_{a}\backslash\left\{  \left\vert V_{a}\right\vert \leq1\right\}  $
away from the free boundary, we know that $\nabla U_{\varepsilon,a}$ converges
almost everywhere to $\nabla V_{a}$. Dominated converges theorem then yields%
\[
\int_{\Omega_{a}}\left[  \left\vert \nabla V_{a}\right\vert ^{2}+\chi_{\left(
-1,1\right)  }\left(  V_{a}\right)  \right]  =\lim_{\varepsilon\rightarrow
0}\int_{\Omega_{a}}\left[  \left\vert \nabla U_{\varepsilon}\right\vert
^{2}+F_{\varepsilon}\left(  U_{\varepsilon}\right)  \right]  \geq2a^{2}%
+\frac{k^{2}}{20}\ln a-C.
\]
This is $\left(  \ref{V}\right)  .$
\end{proof}

\section{Asymptotic analysis of $\left\{  V_{a}\right\}  \label{Sec4}$}

In this section, we show that as $a\rightarrow+\infty,$ up to a subsequence,
$V_{a}$ converges to a solution $W_{k}$ of the free boundary problem \eqref{EN}.

We will also have some information of the asymptotic behavior of $W_{k}$ as
$r$ tends to infinity. We will need the following

\begin{lemma}
\label{mean curvature}Let $u$ be a solution to $\left(  \ref{EN}\right)  $,
with smooth free boundary. Then the mean curvature of the surface
$\partial\left\{  \left\vert u\right\vert <1\right\}  $ is nonnegative, with
respect to the unit normal pointing outwards of $\left\{  \left\vert
u\right\vert <1\right\}  .$
\end{lemma}

\begin{proof}
By \cite[Proposition 2.1]{Wangke}, $\left\vert \nabla u\right\vert \leq1$ in
$\left\{  \left\vert u\right\vert <1\right\}  .$ Hence the maximum of
$\left\vert \nabla u\right\vert $ is achieved at the free boundary, then the
assertion of this lemma follows from \cite[Remark 2]{C1}.
\end{proof}

\begin{lemma}
\label{Asy}Suppose $u$ is a solution to $\left(  \ref{EN}\right)  $ depending
only on $r$ and $\left\vert z\right\vert .$ Assume $\partial_{r}u<0$ and
$\partial_{z}u>0$ in $\Omega=\left\{  \left(  r,z\right)  :z>0\text{ and
}\left\vert u\left(  r,z\right)  \right\vert <1\right\}  .$ Let $r_{0}$ be a
large constant. Suppose that in the region where $r>r_{0},$
\begin{align*}
\partial\Omega\cap\left\{  \left(  r,z\right)  :u\left(  r,z\right)
=1\right\}   &  =\left\{  \left(  r,z\right)  :z=f_{1}\left(  r\right)
\right\}  ,\\
\partial\Omega\cap\left\{  \left(  r,z\right)  :u\left(  r,z\right)
=-1\right\}   &  =\left\{  \left(  r,z\right)  :z=f_{2}\left(  r\right)
\right\}  .
\end{align*}
Then there exist $k>0$ and $b\in\mathbb{R}$, such that
\begin{align*}
f_{1}\left(  r\right)  -k\ln r-b  &  \rightarrow0,\\
f_{2}\left(  r\right)  -k\ln r-b+2  &  \rightarrow0,
\end{align*}
as $r\rightarrow+\infty.$
\end{lemma}

\begin{proof}
We write
\begin{equation}
\frac{rf_{1}^{\prime}\left(  r\right)  }{\sqrt{1+f_{1}^{\prime2}}}=\int%
_{r_{0}}^{r}\left[  \frac{rf_{1}^{\prime}\left(  r\right)  }{\sqrt
{1+f_{1}^{\prime2}}}\right]  ^{\prime}ds+\frac{r_{0}f_{1}^{\prime}\left(
r_{0}\right)  }{\sqrt{1+f_{1}^{\prime2}\left(  r_{0}\right)  }}:=a_{1}\left(
r\right)  . \label{s0}%
\end{equation}
Applying Lemma \ref{mean curvature}, we get
\[
\left[  \frac{rf_{1}^{\prime}\left(  r\right)  }{\sqrt{1+f_{1}^{\prime2}}%
}\right]  ^{\prime}\geq0.
\]
Hence $a_{1}\left(  \cdot\right)  $ is positive and monotone increasing.
Similarly, applying Lemma \ref{mean curvature}, we have
\begin{equation}
\frac{rf_{2}^{\prime}\left(  r\right)  }{\sqrt{1+f_{2}^{\prime2}}}=\int%
_{r_{0}}^{r}\left[  \frac{rf_{2}^{\prime}\left(  r\right)  }{\sqrt
{1+f_{2}^{\prime2}}}\right]  ^{\prime}ds+\frac{r_{0}f_{2}^{\prime}\left(
r_{0}\right)  }{\sqrt{1+f_{2}^{\prime2}\left(  r_{0}\right)  }}:=a_{2}\left(
r\right)  , \label{s2}%
\end{equation}
where $a_{2}$ is monotone decreasing. On the other hand, using the
monotonicity of $u,$ we can show that as $r$ tends to infinity, $u$ behaves
locally like suitable vertical translation of the one dimensional profile
$\mathcal{H}$. (By the De Giorgi type classification result, see
\cite{Wangke}). This together with $\left(  \ref{s0}\right)  ,$ $\left(
\ref{s2}\right)  $ imply that
\[
\lim_{r\rightarrow+\infty}a_{1}\left(  r\right)  =\lim_{r\rightarrow+\infty
}a_{2}\left(  r\right)  =k\in\left(  0,+\infty\right)  .
\]

Now we can write
\begin{align*}
f_{1}^{\prime}\left(  r\right)   &  =\frac{a_{1}}{r}\frac{1}{\sqrt
{1-r^{-2}a_{1}^{2}}}:=\frac{a_{1}\left(  r\right)  }{r}+\eta_{1}\left(
r\right)  ,\\
f_{2}^{\prime}\left(  r\right)   &  =\frac{a_{2}}{r}\frac{1}{\sqrt
{1-r^{-2}a_{2}^{2}}}:=\frac{a_{2}\left(  r\right)  }{r}+\eta_{2}\left(
r\right)  ,
\end{align*}
where $\eta_{i}\left(  r\right)  =O\left(  r^{-3}\right)  $ as $r\rightarrow
+\infty.$ Therefore
\begin{align}
f_{1}\left(  r\right)  -f_{2}\left(  r\right)   &  =\int_{r_{0}}^{r}%
\frac{a_{1}\left(  s\right)  -a_{2}\left(  s\right)  }{s}ds+\int_{r_{0}}%
^{r}\left(  \eta_{1}\left(  s\right)  -\eta_{2}\left(  s\right)  \right)
ds\nonumber\\
&  +f_{1}\left(  r_{0}\right)  -f_{2}\left(  r_{0}\right)  . \label{f}%
\end{align}
Observe that $\lim_{r\rightarrow+\infty}\left(  f_{1}\left(  r\right)
-f_{2}\left(  r\right)  \right)  =2.$ Then $\left(  \ref{f}\right)  $ together
with the fact that $a_{1}\left(  r\right)  \leq k$ and $a_{2}\left(  r\right)
\geq k$ tell us that
\[
\int_{r_{0}}^{+\infty}\frac{k-a_{1}\left(  r\right)  }{s}<+\infty,\text{ }%
\int_{r_{0}}^{+\infty}\frac{a_{2}\left(  r\right)  -k}{s}<+\infty.
\]
This in turn implies the existence of $b$ such that
\[
f_{1}\left(  r\right)  -k\ln r-b\rightarrow0\text{, as }r\rightarrow+\infty.
\]
The proof is thus completed.
\end{proof}

Our next task is to show that the distance of the free boundary of $V_{a}$ to
the origin $O$ is uniformly bounded.

\begin{lemma}
\label{dis}Let $\digamma_{a}^{\pm}$ be defined by $\left(  \ref{CF}\right)  .$
There exists a constant $C$ independent of $a,$ such that
\[
\operatorname{dist}\left(  O,\digamma_{a}^{\pm}\right)  \leq C.
\]

\end{lemma}

\begin{proof}
Assume to the contrary that the conclusion of the lemma were not true. There
are three possibilities.

Case 1. $\digamma_{a}^{-}\cap\left\{  \left(  r,z\right)  :z=0\right\}
=\varnothing.$

In this case, moving plane argument tells us that $V_{a}$ is the trivial one
dimensional (only depends on $z$ variable) solution. To be more precise, let
us consider the family of trivial solutions $\mathcal{H}\left(  z-\beta
\right)  $ where $\beta$ is a parameter. We start with $\beta<0 ($sufficiently
small$)$ and increase $\beta$ continuously until their free boundaries touch
at some point. Monotonicity of the solution implies that the free boundary of
$\mathcal{H}$ and $V_{a}$ must touch inside $\Omega_{a}.$ Maximum principle
then tells us that $V_{a}$ is the trivial one dimensional solution. But this
contradicts with the energy estimate $\left(  \ref{V}\right)  .$ We remark
that actually if the free boundary $\digamma_{a}^{-}$ intersects with
$L_{1,a}$ at a point $\left(  a,z_{0}\right)  $ with $z_{0}%
<k\operatorname{arccosh}\left(  k^{-1}a\right)  -1,$ then at this intersection
point, they must touch tangentially (see \cite{KKS,K}).

Case 2. $\digamma_{a}^{+}\cap\left\{  \left(  r,z\right)  :r=0\right\}
=\varnothing.$

Subcase 1. There exists a universal constant $C$ such that
\[
\digamma_{a}^{+}\subset\left\{  \left(  r,z\right)  :a-C<r<a\right\}  .
\]
Then using the fact that $\left\vert \nabla V_{a}\right\vert $ is uniformly
bounded in $a$, we estimate
\begin{align*}
J\left(  V_{a}\right)   &  =\int\left[  \left\vert \nabla V_{a}\right\vert
^{2}+\chi_{\left(  -1,1\right)  }\left(  V_{a}\right)  \right] \\
&  \leq Ca\ln a,
\end{align*}
which contradicts with the energy estimate $\left(  \ref{V}\right)  .$

Subcase 2. There is a sequence $\left\{  a_{i}\right\}  $ tending to infinity
and a sequence of points $P_{i}\in\digamma_{a_{i}}^{+}\cap\left\{  \left(
r,z\right)  :z=0\right\}  ,$ with $\left\vert P_{i}\right\vert $ also tending
to infinity, such that $\operatorname{dist}\left(  P_{i},\left(
a_{i},0\right)  \right)  \rightarrow+\infty.$

In this case, from the construction in \cite{Liu}, we know that there is a
family of solutions $\bar{u}_{\lambda}$ to the free boundary problem $\left(
\ref{EN}\right)  $ whose nodal set is close to the family of rescaled
catenoids $z=\lambda\operatorname{arccosh}\left(  \lambda^{-1}r\right)  ,$
where $\lambda$ is a (large) parameter. Moving plane type arguments based on
$\bar{u}_{\lambda}$ then tell us that we can touch $V_{a_{i}}$ inside
$\Omega_{a_{i}}$ with some $\bar{u}_{\lambda}.$ This contradicts with the
maximum principle.

Case 3. $\digamma_{a}^{+}\cap\left\{  \left(  r,z\right)  :z=0\right\}
=\varnothing$ and $\digamma_{a}^{-}\cap\left\{  \left(  r,z\right)
:r=0\right\}  =\varnothing.$

Subcase 1. $\operatorname{dist}\left(  O,\digamma_{a_{i}}^{+}\right)
\rightarrow+\infty,$ for a sequence $\left\{  a_{i}\right\}  .$

Let $P_{a}$ be the intersection of $\digamma_{a}^{+}$ with the $z$ axis. Then
the sequence of functions $h_{a_{i}}\left(  \cdot\right)  =V_{a_{i}}\left(
\cdot-P_{a_{i}}\right)  $ converges in $C^{0,\alpha}$ to a function
$h_{\infty}.$ $h_{\infty}$ is a variational solution in the sense of $\left(
\ref{va}\right)  .$ Each $V_{a_{i}}$ is nondegenerated, hence the free
boundary point of $h_{\infty}$ is also nondegenerated. Blow up analysis then
tells us that the free boundary is regular. From De Giorgi type results, we
infer that $h_{\infty}$ is a one dimensional solution. That is, $h_{\infty
}\left(  r,z\right)  =\mathcal{H}\left(  z+1\right)  .$ This contradicts with
the monotonicity of $V_{a_{i}}$ in the $z$ direction.

Subcase 2. $\operatorname{dist}\left(  O,\digamma_{a}^{-}\right)
\rightarrow+\infty.$

In this case, we can proceed similarly as Subcase 1. We omit the details.
\end{proof}

With Lemma \ref{dis} understood, we state the following

\begin{proposition}
\label{Asym}For each $k\in\left(  0,+\infty\right)  ,$ there exists a solution
$W_{k}$ to the free boundary problem $\left(  \ref{EN}\right)  $ whose nodal
set $\left\{  \left(  r,z\right)  :z=f\left(  r\right)  \right\}  $ has the
following asymptotic behavior: There exists a constant $b_{k}$ such that
\begin{equation}
f\left(  r\right)  -k\ln r-b_{k}\rightarrow0,\text{ as }r\rightarrow+\infty.
\end{equation}

\end{proposition}

Before starting the proof, let us establish the following

\begin{lemma}
\label{y}Fix a constant $\bar{k}>0$ with $\bar{k}\neq k.$ Suppose $b$ and
$a/b$ is large. Let $\xi$ be a $C^{1}$ function satisfying $\xi\left(
b\right)  =\bar{k}\ln b$ and $\xi\left(  a\right)  =k\ln a.$ Then%
\[
\int_{b}^{a}\sqrt{1+\xi^{\prime2}}rdr\geq\frac{1}{2}a^{2}-\frac{1}{2}%
b^{2}+\frac{1}{2}\frac{\left(  k\ln a-\bar{k}\ln b\right)  ^{2}}{\ln a-\ln
b}-C,
\]
where $C$ does not depend on $a,b.$
\end{lemma}

\begin{proof}
The points $\left(  b,\bar{k}\ln b\right)  $ and $\left(  a,k\ln a\right)  $
are on the catenoid $z=\sigma\operatorname{arccosh}\left(  \sigma
^{-1}r\right)  +d:=\eta\left(  r\right)  ,$ where $\sigma,d$ satisfies
\[
\left\{
\begin{array}
[c]{l}%
\sigma\operatorname{arccosh}\left(  \sigma^{-1}b\right)  +d=\bar{k}\ln b,\\
\sigma\operatorname{arccosh}\left(  \sigma^{-1}a\right)  +d=k\ln a.
\end{array}
\right.
\]
The existence of $\sigma$ is guaranteed by the assumption that $b$ and $a/b$
is large. $\sigma$ has the estimate
\[
\sigma=\frac{k\ln a-\bar{k}\ln b}{\ln a-\ln b}+O\left(  \frac{1}{\left(  \ln
a-\ln b\right)  b^{2}}\right)  .
\]
Using this, we then compute
\begin{align*}
\int_{b}^{a}\sqrt{1+\eta^{\prime2}\left(  r\right)  }rdr  &  =\sigma\int%
_{\bar{k}\ln b-d}^{k\ln a-d}\cosh^{2}\left(  \sigma^{-1}z\right)  dz\\
&  \geq\frac{1}{2}\left(  a^{2}-b^{2}\right)  +\frac{\sigma}{2}\left(  \bar
{k}\ln b-k\ln a\right)  -C\\
&  \geq\frac{1}{2}a^{2}-\frac{1}{2}b^{2}+\frac{1}{2}\frac{\left(  k\ln
a-\bar{k}\ln b\right)  ^{2}}{\ln a-\ln b}-C,
\end{align*}
provided that $b$ and $a/b$ is large. The desired estimate of this lemma then
follows from the fact that the catenoid is a(parametric in this case) minimal surface.
\end{proof}

\begin{proof}
[Proof of Proposition \ref{Asym}]We would like to get a uniform estimate for
the sequence of solutions $V_{a}$ independent on $a.$ Once we have this
estimate, we can let $a\rightarrow+\infty$ and get a solution $W_{k}$ with
desired asymptotic behavior at infinity.

By Lemma \ref{dis}, a subsequence of $\left\{  V_{a}\right\}  $ converges in
$C_{loc}^{0,\alpha}\left(  \mathbb{R}^{3}\right)  $ to a solution $W$ of
$\left(  \ref{EN}\right)  .$ Since $V_{a}$ is monotone, $W$ is also
monotone(in both $r$ and $z$ direction). By Lemma \ref{Asy}, there exists
$\bar{k}>0$ and $b_{k}\in\mathbb{R}$ such that
\[
f\left(  r\right)  -\bar{k}\ln r-b_{k}\rightarrow0,\text{ as }r\rightarrow
+\infty.
\]
It suffices to prove that $\bar{k}=k.$

We argue by contradiction and assume $\bar{k}\neq k.$ We would like to show
that for $a$ sufficiently large, the energy of $V_{a}$ satisfies
\[
J\left(  V_{a}\right)  -\lim\sup_{\varepsilon\rightarrow0}c^{\ast}>0.
\]

Fix a large constant $A_{1}$ such that in the region $\mathbb{R}^{3}\backslash
B_{A_{1}},$ the nodal set of $W_{k}$ is close to $z=\bar{k}\ln r+b_{k}.$ Then
we can estimate%
\[
\int_{B_{A_{1}}}\left[  \left\vert \nabla V_{a}\right\vert ^{2}+\chi_{\left(
-1,1\right)  }\left(  V_{a}\right)  \right]  =2A_{1}^{2}+2\bar{k}^{2}\ln
A_{1}+O\left(  1\right)  .
\]
On the other hand, the energy outside the ball $B_{A_{1}}$ satisfies
\[
\int_{\Omega_{a}\backslash B_{A_{1}}}\left[  \left\vert \nabla V_{a}%
\right\vert ^{2}+\chi_{\left(  -1,1\right)  }\left(  V_{a}\right)  \right]
\geq2\int_{\Omega_{a}\backslash B_{A_{1}}}\left\vert \nabla V_{a}\right\vert
=2\int_{-1}^{1}\left\vert \left\{  V_{a}=s\right\}  \cap\left(  \Omega
_{a}\backslash B_{A_{1}}\right)  \right\vert ds.
\]
Using Lemma \ref{y},
\[
\left\vert \left\{  V_{a}=s\right\}  \cap\left(  \Omega_{a}\backslash
B_{A_{1}}\right)  \right\vert \geq\frac{1}{2}a^{2}-\frac{1}{2}A_{1}^{2}%
+\frac{1}{2}\frac{\left(  k\ln a-\bar{k}\ln A_{1}\right)  ^{2}}{\ln a-\ln
A_{1}}-C.
\]
Therefore, recalling the upper bound estimate of $c^{\ast}$(Lemma
\ref{cupper})$,$ we get
\begin{align*}
J\left(  V_{a}\right)  -\lim\sup_{\varepsilon\rightarrow0}c^{\ast}  &
\geq2\frac{\left(  k\ln a-\bar{k}\ln A_{1}\right)  ^{2}}{\ln a-\ln A_{1}%
}+2\bar{k}^{2}\ln A_{1}-2k^{2}\ln a-C\\
&  =2\frac{\left(  k-\bar{k}\right)  ^{2}\ln a\ln A_{1}}{\ln a-\ln A_{1}}-C\\
&  \geq2\left(  k-\bar{k}\right)  ^{2}\ln A_{1}-C>0,
\end{align*}
provided that $A_{1}$ is large enough. This is a contradiction
\end{proof}

\section{Asymptotic analysis of $\left\{  W_{k}\right\}  $\label{Sec5}}

In this section, we first show that as $k\rightarrow0,$ $W_{k}$ converges to
the function $\left\vert z\right\vert -1$ in the region $\left\{  \left(
r,z\right)  :\left\vert z\right\vert <2\right\}  .$ Then we can perform a
rescaling on $W_{k}$ and prove that the resulted sequence of functions
converges to the desired solution of the one phase free boundary problem. Some
computations in this section are similar to those in \cite{Liu2}. A main step
in the argument is the analysis of the asymptotic behavior of $W_{k}.$

We set
\[
\digamma=\partial\left\{  \left\vert W_{k}\right\vert <1\right\}
,\digamma^{\pm}=\digamma\cap\left\{  W_{k}=\pm1\right\}  .
\]
Due to monotonicity of the solution, $\digamma^{-}$ is represented by the
graph of a function $p_{k}:$
\[
\digamma^{-}=\left\{  \left(  r,z\right)  :z=p_{k}\left(  r\right)  \right\}
.
\]

\begin{proposition}
\label{uni}For $k\in\left(  0,1\right)  ,$ there exist $b_{k},$ such that
$\left\vert b_{k}\right\vert \leq C$ and
\[
\left\vert p_{k}\left(  r\right)  -k\ln r-b_{k}\right\vert \leq Cr^{-1},\text{
for all }r>r_{0},
\]
where $r_{0},C$ are certain constants independent of $k.$
\end{proposition}

For notational convenience, we will not write the subscript $k$ if no
confusion will arise. The main difficulty in the proof of Proposition
\ref{uni} is that although $p^{\prime\prime}\left(  r\right)  \rightarrow0$ as
$r\rightarrow+\infty,$ which follows from the regularity theory of
Kinderlehrer-Nirenberg\cite{KN}, a priori we do not have any decay information
on $p^{\prime\prime}.$ We use $\left(  l,s\right)  $ to denote the Fermi
coordinate around the curve $\digamma^{-}.$ Explicitly, for a given point, the
relation between its $\left(  l,s\right)  $ and $\left(  r,z\right)  $
coordinate is given by
\[
\left\{
\begin{array}
[c]{c}%
r=l-\frac{p^{\prime}}{\sqrt{1+p^{\prime2}}}s,\\
z=p+\frac{1}{\sqrt{1+p^{\prime2}}}s,
\end{array}
\right.
\]
where $p,p^{\prime}$ is evaluated at $l.$ Since $p^{\prime\prime}$ is small,
this Fermi coordinate is well defined in a large (depending on $p^{\prime
\prime}$) tubular neighbourhood of $\digamma^{-}.$ Set
\[
\Gamma_{h}:=\left\{  X+h\nu\left(  X\right)  :X\in\digamma^{-}\right\}  ,
\]
where $\nu$ is a unit normal of $\digamma^{-},$ pointing upwards. Then
$\Gamma_{0}=\digamma^{-}.$ We also know that $\digamma^{+}$ can be written as
$\Gamma_{h},$ for a function $h$ close to $2.\ $

Now we define an approximate solution $\bar{W}$ in terms of the Fermi
coordinate as
\[
\bar{W}\left(  l,s\right)  =\frac{s}{1+f\left(  l\right)  }-1,
\]
where $f=\frac{h}{2}-1.$ Then $\bar{W}=-1$ on $\digamma^{-},$ and $\bar{W}=1$
on $\digamma^{+}.$ We write $W$ as $\bar{W}+\phi$ and want to estimate $\phi.$

It will be important to estimate the error of the approximate solution
$\bar{W}.$ We use $H_{M}$ to denote the mean curvature of a surface $M.$ Then
we compute the Laplacian of $\bar{W}$ in the Fermi coordinate
\begin{align*}
\Delta\bar{W}  &  =\Delta_{\Gamma_{s}}\bar{W}+\partial_{s}^{2}\bar
{W}-H_{\Gamma_{s}}\partial_{s}\bar{W}\\
&  =\Delta_{\Gamma_{s}}\bar{W}-\frac{H_{\Gamma_{s}}}{1+f}.
\end{align*}
Let us use $k_{i},i=1,2,$ to denote the principle curvatures of $\digamma
^{-}:$
\[
k_{1}=\frac{p^{\prime\prime}}{\left(  p^{\prime2}+1\right)  ^{\frac{3}{2}}},
\quad k_{2}=\frac{p^{\prime}}{r\sqrt{1+p^{\prime2}}}.
\]
Then the mean curvature of $\digamma^{-}$ at the point $\left(  r,z\right)  $
is
\[
H=\frac{1}{r}\left(  \frac{rp^{\prime}\left(  r\right)  }{\sqrt{1+p^{\prime
2}\left(  r\right)  }}\right)  ^{\prime}.
\]
We will set $\left\vert A\right\vert ^{2}=k_{1}^{2}+k_{2}^{2}$ and
\[
t=\frac{s}{1+f\left(  l\right)  }-1.
\]

\begin{lemma}
\label{l1}The Laplacian operator on $\Gamma_{0}$ acting on $\bar{W}$
satisfies
\[
\Delta_{\Gamma_{0}}\bar{W}=-t\Delta_{\Gamma_{0}}f+I_{1},
\]
where
\[
I_{1}=-tf\Delta_{\Gamma_{0}}f+\Delta_{\Gamma_{0}}\left(  \frac{sf^{2}}%
{1+f}\right)  .
\]

\end{lemma}

\begin{proof}
We can write
\begin{align*}
\bar{W}  &  =\frac{s}{1+f\left(  l\right)  }-1=s\left(  1-f+\frac{f^{2}}%
{1+f}\right)  -1\\
&  =s-sf+\frac{sf^{2}}{1+f}-1\text{. }%
\end{align*}
We then compute
\[
\Delta_{\Gamma_{0}}\bar{W}=-s\Delta_{\Gamma_{0}}f+\Delta_{\Gamma_{0}}\left(
\frac{sf^{2}}{1+f}\right)  .
\]
Inserting the relation $s=t\left(  1+f\right)  $ into the left hand side, we
get%
\begin{align*}
\Delta_{\Gamma_{0}}\bar{W}  &  =-t\left(  1+f\right)  \Delta_{\Gamma_{0}%
}f+\Delta_{\Gamma_{0}}\left(  \frac{sf^{2}}{1+f}\right) \\
&  =-t\Delta_{\Gamma_{0}}f-tf\Delta_{\Gamma_{0}}f+\Delta_{\Gamma_{0}}\left(
\frac{sf^{2}}{1+f}\right)  .
\end{align*}
This finishes the proof.
\end{proof}

\begin{lemma}
\label{l2}We have the following expansion for the mean curvature of
$\Gamma_{s}:$%
\[
\frac{H_{\Gamma_{s}}}{1+f}=\frac{H_{\Gamma_{0}}}{1+f}+t\left\vert A\right\vert
^{2}+I_{2},
\]
where
\[
I_{2}=\frac{1}{1+f}\sum\limits_{i=1}^{2}\frac{s^{2}k_{i}^{3}}{1-sk_{i}}.
\]

\end{lemma}

\begin{proof}
The mean curvature of the surface $\Gamma_{s}$ has the form (see \cite{DW}):
\[
H_{\Gamma_{s}}=\sum\limits_{i=1}^{2}\frac{k_{i}}{1-sk_{i}}=H_{\Gamma_{0}}%
+\sum\limits_{i=1}^{2}sk_{i}^{2}+\sum\limits_{i=1}^{2}\frac{s^{2}k_{i}^{3}%
}{1-sk_{i}}.
\]
Hence
\begin{align*}
\frac{H_{\Gamma_{s}}}{1+f}  &  =\frac{H_{\Gamma_{0}}}{1+f}+\frac{\left\vert
A\right\vert ^{2}}{1+f}s+\frac{1}{1+f}\sum\limits_{i=1}^{2}\frac{s^{2}%
k_{i}^{3}}{1-sk_{i}}\\
&  =\frac{H_{\Gamma_{0}}}{1+f}+t\left\vert A\right\vert ^{2}+\frac{1}{1+f}%
\sum\limits_{i=1}^{2}\frac{s^{2}k_{i}^{3}}{1-sk_{i}}.
\end{align*}
This finishes the proof.
\end{proof}

The function $\phi$ satisfies $\phi=0$ on $\digamma^{\pm}.$ By Lemma \ref{l1}
and Lemma \ref{l2}, we have
\begin{align}
\Delta\phi &  =-\Delta\bar{W}=\frac{H_{\Gamma_{0}}}{1+f}+\left(
\Delta_{\Gamma_{0}}f+\left\vert A\right\vert ^{2}\right)  t\nonumber\\
&  -I_{1}+I_{2}+\Delta_{\Gamma_{0}}\bar{W}-\Delta_{\Gamma_{s}}\bar{W}
\quad\text{ in }\left\{  \left\vert W_{k}\right\vert <1\right\}  . \label{fi}%
\end{align}
Our next purpose is to analyze the boundary condition $\left\vert \nabla
W\right\vert =1.$ We use $\mathfrak{g}_{s}^{i,j}=\mathfrak{g}^{i,j}$ to denote
the entries of inverse matrix of the metric tensor on $\Gamma_{s}.$

\begin{lemma}
\label{l3}On $\Gamma_{0},$ we have%
\[
\partial_{t}\phi-f=I_{3,-},
\]
where
\begin{equation}
I_{3,-}=\frac{f^{2}}{2}-\frac{\left(  \partial_{t}\phi\right)  ^{2}}%
{2}-\mathfrak{g}_{0}^{1,1}\frac{\left(  t+1\right)  ^{2}f^{\prime2}}{2}.
\label{3-}%
\end{equation}
Similarly,
\[
\partial_{t}\phi-f=I_{3,+},\text{ on }\Gamma_{h},
\]
where
\begin{equation}
I_{3,+}=-\frac{1}{2}\left(  1+\mathfrak{g}^{1,1}h^{\prime2}\right)  \left(
\partial_{t}\phi\right)  ^{2}+\frac{\mathfrak{g}_{h}^{1,1}h^{\prime}}%
{1+f}\partial_{t}\phi+\frac{1}{2}f^{2}-\frac{1}{2}\mathfrak{g}_{h}%
^{1,1}\left(  \left(  t+1\right)  f^{^{\prime}}\right)  ^{2}.\text{ }
\label{3+}%
\end{equation}

\end{lemma}

\begin{proof}
We have
\begin{equation}
\left\vert \nabla\left(  \bar{W}+\phi\right)  \right\vert ^{2}=\left(
\partial_{s}\bar{W}+\partial_{s}\phi\right)  ^{2}+\mathfrak{g}^{1,1}\left(
\partial_{l}\bar{W}+\partial_{l}\phi\right)  ^{2}. \label{G1}%
\end{equation}
Since $\partial_{s}\bar{W}=\frac{1}{1+f},$ $\partial_{l}\bar{W}=-\frac
{sf^{\prime}}{\left(  1+f\right)  ^{2}},$ and $\partial_{l}\phi=0$ on
$\Gamma_{0},$ we obtain from $\left\vert \nabla W\right\vert =1$ that
\begin{equation}
\left(  \partial_{s}\phi\right)  ^{2}+\frac{2}{1+f}\partial_{s}\phi+\frac
{1}{\left(  1+f\right)  ^{2}}+\mathfrak{g}_{0}^{1,1}\frac{s^{2}f^{\prime2}%
}{\left(  1+f\right)  ^{4}}=1,\text{ on }\Gamma_{0}. \label{G2}%
\end{equation}
Observe that $\partial_{s}\phi=\frac{\partial_{t}\phi}{1+f}.$ Hence
\[
\left(  \partial_{t}\phi\right)  ^{2}+2\partial_{t}\phi+\mathfrak{g}_{0}%
^{1,1}\frac{s^{2}f^{\prime2}}{\left(  1+f\right)  ^{2}}=2f+f^{2},\text{ on
}\Gamma_{0}.
\]
This is $\left(  \ref{3-}\right)  .$

On $\Gamma_{h},$ since $\phi\left(  l,h\left(  l\right)  \right)  =0,$ we have
$\partial_{l}\phi=-\partial_{s}\phi h^{\prime}.$ Hence from $\left(
\ref{G1}\right)  $ and $\left\vert \nabla W\right\vert =1,$ we deduce%
\[
\left(  1+\mathfrak{g}_{h}^{1,1}h^{\prime2}\right)  \left(  \partial_{s}%
\phi\right)  ^{2}+\left(  \frac{2}{1+f}-2\mathfrak{g}_{h}^{1,1}\partial
_{l}\bar{W}h^{\prime}\right)  \partial_{s}\phi+\frac{1}{\left(  1+f\right)
^{2}}+\mathfrak{g}_{h}^{1,1}\frac{s^{2}f^{\prime2}}{\left(  1+f\right)  ^{4}%
}=1.
\]
This is $\left(  \ref{3+}\right)  .$
\end{proof}

It is expected that the functions $f$ and $p^{\prime\prime}$ decays like
$O\left(  l^{-2}\right)  $ as $l\rightarrow+\infty.$ To prove this, we need to
work in a suitable (exponentially weighted, rather than algebraically
weighted) functional spaces. Fix an $\alpha\in\left(  0,1\right)  .$

\begin{definition}
For $\mu=0,1,2,$ $\beta\geq0,$ the space $\mathcal{B}_{\beta,\mu}$ consists of
those functions $\eta=\eta\left(  l\right)  ,l\in\lbrack0,+\infty),$ such
that
\[
\left\Vert \eta\right\Vert _{\beta,\mu}:=\sup_{l}\left[  e^{\beta l}\left\Vert
\eta\right\Vert _{C^{\mu,\alpha}\left(  [l,l+1]\right)  }\right]  <+\infty.
\]

\end{definition}

\begin{definition}
For $\mu=0,1,2,$ $\beta\geq0$, the space $\mathcal{B}_{\beta,\mu;\ast}$
consists of those functions $\varphi=\phi\left(  t,l\right)  ,$ $\left(
t,l\right)  \in\left[  -1,1\right]  \times\lbrack0,+\infty),$ such that
\[
\left\Vert \varphi\right\Vert _{\beta,\mu;\ast}:=\sup_{\left(  t,l\right)
}\left[  e^{\beta l}\left\Vert \varphi\right\Vert _{C^{\mu,\alpha}\left(
\left[  -1,1\right]  \times\lbrack l,l+1]\right)  }\right]  <+\infty.
\]

\end{definition}

\begin{lemma}
\label{estimate1}Let $\delta>0$ be a fixed small constant. Assume $\beta
\in\left[  0,\delta\right]  .$ Suppose $\eta\in\mathcal{B}_{\beta,0;\ast}$ and
$\Phi\in\mathcal{B}_{\beta,2;\ast}$ satisfying
\[
\left\{
\begin{array}
[c]{l}%
\partial_{t}^{2}\Phi+\partial_{l}^{2}\Phi+\frac{1}{l}\partial_{l}\Phi
=\eta,\text{ }\left[  -1,1\right]  \times\lbrack0,+\infty),\\
\Phi\left(  t,l\right)  =0\text{, for }t=\pm1.
\end{array}
\right.
\]
Then $\left\Vert \Phi\right\Vert _{\beta,2;\ast}\leq C\left\Vert
\eta\right\Vert _{\beta,0;\ast}.$
\end{lemma}

The proof of Lemma \ref{estimate1} follows from standard arguments, see for
instance Lemma 5.1 of \cite{Ma}, where a more complicated situation for the
Allen-Cahn equation is studied. We omit the details.

The solution $W$ resembles the one dimensional profile only when $r$ is large,
say $r>r_{0}.$ Therefore the function $\phi$ is only well defined in the
region $r>r_{0}.$ Note that $r_{0}$ can be chosen to be independent of $k.$ We
introduce a cutoff function $\zeta$ such that
\[
\zeta\left(  s\right)  =\left\{
\begin{array}
[c]{c}%
1,s>1,\\
0,s<0.
\end{array}
\right.
\]
Slightly abusing the notation, we still write the function $\phi$ in the
$\left(  t,l\right)  $ coordinate as $\phi\left(  t,l\right)  .$ For
$a>r_{0},$ let $\Psi_{a}\left(  t,l\right)  :=\zeta_{a}\left(  l\right)
\phi\left(  t,l\right)  $ and $\bar{f}_{a}=\zeta_{a}f,$ where $\zeta
_{a}\left(  l\right)  =\zeta\left(  l-a\right)  .$ Using $\left(
\ref{fi}\right)  $ and Lemma \ref{l3}, we find that $\Psi_{a}$ satisfies
\begin{equation}
\left\{
\begin{array}
[c]{l}%
\Delta\Psi_{a}=\zeta_{a}H_{\Gamma_{0}}+\left(  \Delta_{\Gamma_{0}}f+\left\vert
A\right\vert ^{2}\right)  t\zeta_{a}+P,\text{ }\left(  t,l\right)  \in\left[
-1,1\right]  \times\lbrack0,+\infty),\\
\Psi_{a}\left(  \pm1,l\right)  =0,\\
\partial_{t}\Psi_{a}-\bar{f}_{a}=\gamma_{-},\text{for }t=-1,\\
\partial_{t}\Psi_{a}-\bar{f}_{a}=\gamma_{+},\text{ for }t=1.
\end{array}
\right.  \label{linear2}%
\end{equation}
Here $P$ is a perturbation term and explicitly,
\[
P=\left(  I_{2}-I_{1}+\Delta_{\Gamma_{0}}\bar{W}-\Delta_{\Gamma_{s}}\bar
{W}\right)  \zeta_{a}+2\nabla\zeta_{a}\nabla\phi+\Delta\zeta_{a}\phi,
\]
and $\gamma_{+}=\zeta_{a}I_{3,+},$ $\gamma_{-}=\zeta_{a}I_{3,-}.$

We need the following linear theory.

\begin{lemma}
\label{estimate2}Suppose $\beta\in\left[  0,\delta\right]  ,$ with $\delta>0$
being small. Assume $\eta^{\pm}\in\mathcal{B}_{\beta,1},$ $g_{1},g_{2}%
\in\mathcal{B}_{\beta,0},$ and $\vartheta\in\mathcal{B}_{\beta,0,\ast}.$ If
$\Phi$ satisfies
\begin{equation}
\left\{
\begin{array}
[c]{l}%
\partial_{t}^{2}\Phi+\partial_{l}^{2}\Phi+\frac{1}{l}\partial_{l}\Phi
=g_{1}+t\Delta_{\Gamma_{0}}g_{2}+\vartheta,\text{ }\left(  t,l\right)
\in\left[  -1,1\right]  \times\lbrack0,+\infty).\\
\Phi\left(  \pm1,l\right)  =0,\\
\partial_{t}\Phi-g_{2}=\eta_{-},\text{for }t=-1,\\
\partial_{t}\Phi-g_{2}=\eta_{+},\text{ for }t=1.
\end{array}
\right.  \label{L}%
\end{equation}
Then
\begin{align*}
\left\Vert g_{1}\right\Vert _{\beta,0}  &  \leq C\left\Vert \eta_{+}-\eta
_{-}\right\Vert _{\beta,1}+C\left\Vert \vartheta\right\Vert _{\beta,0,\ast},\\
\left\Vert g_{2}\right\Vert _{\beta,2}  &  \leq C\left\Vert \eta_{+}+\eta
_{-}\right\Vert _{\beta,1}+C\left\Vert \vartheta\right\Vert _{\beta,0,\ast}.
\end{align*}

\end{lemma}

\begin{proof}
The proof of this lemma is similar as \cite[Proposition 17 ]{Liu}, using
Fourier transform. We sketch the proof for completeness. For each $\xi
\in\mathbb{R}^{2},$ let $q_{1,\xi},q_{2,\xi}$ solve
\[
\left\{
\begin{array}
[c]{l}%
q_{1,\xi}^{\prime\prime}\left(  t\right)  -\left\vert \xi\right\vert
^{2}q_{1,\xi}\left(  t\right)  =1,\\
q_{1,\xi}\left(  -1\right)  =q_{1,\xi}\left(  1\right)  =0,
\end{array}
\right.
\]
and
\[
\left\{
\begin{array}
[c]{l}%
q_{2,\xi}^{\prime\prime}\left(  t\right)  -\left\vert \xi\right\vert
^{2}q_{2,\xi}\left(  t\right)  =t,\\
q_{2,\xi}\left(  -1\right)  =q_{2,\xi}\left(  1\right)  =0.
\end{array}
\right.
\]
Explicitly, $q_{1,\xi}$ and $q_{2,\xi}$ are given by
\begin{align*}
q_{1,\xi}\left(  t\right)   &  =\frac{\cosh\left(  \left\vert \xi\right\vert
t\right)  }{\left\vert \xi\right\vert ^{2}\cosh\left\vert \xi\right\vert
}-\frac{1}{\left\vert \xi\right\vert ^{2}},\\
q_{2,\xi}\left(  t\right)   &  =\frac{\sinh\left(  \left\vert \xi\right\vert
t\right)  }{\left\vert \xi\right\vert ^{2}\sinh\left\vert \xi\right\vert
}-\frac{t}{\left\vert \xi\right\vert ^{2}}.
\end{align*}

We first deal with the case of $\vartheta=0.$ Taking Fourier transform in
$\left(  \ref{L}\right)  $in $\mathbb{R}^{2}$ with respect to the $z_{1}%
,z_{2}$, where $l=\sqrt{z_{1}^{2}+z_{2}^{2}},$ variable, we are lead to
\[
\left\{
\begin{array}
[c]{l}%
\partial_{t}^{2}\hat{\Phi}-\left\vert \xi\right\vert ^{2}\hat{\Phi}=\hat
{g}_{1}+\hat{g}_{2}t,\text{ }t\in\left[  -1,1\right]  ,\\
\hat{\Phi}\left(  -1,\xi\right)  =\hat{\Phi}\left(  1,\xi\right)  =0,\\
\partial_{t}\hat{\Phi}\left(  -1,\xi\right)  -\hat{g}_{2}\left(  \xi\right)
=\hat{\gamma}_{-1}\left(  \xi\right)  ,\\
\partial_{t}\hat{\Phi}\left(  1,\xi\right)  -\hat{g}_{2}\left(  \xi\right)
=\hat{\gamma}_{1}\left(  \xi\right)  \text{.}%
\end{array}
\right.
\]
It follows that the function $\hat{\Phi}$ has the form
\[
\hat{\Phi}\left(  t,\xi\right)  =\hat{g}_{1}\left(  \xi\right)  q_{1,\xi
}\left(  t\right)  +\left(  \Delta_{\Gamma_{0}}g_{2}\right)  ^{\symbol{94}%
}\left(  \xi\right)  q_{2,\xi}\left(  t\right)  .
\]
In view of the boundary condition at $t=\pm1,$ we get
\[
\left\{
\begin{array}
[c]{l}%
\hat{g}_{1}\left(  \xi\right)  q_{1,\xi}^{\prime}\left(  -1\right)  +\left(
\Delta_{\Gamma_{0}}g_{2}\right)  ^{\symbol{94}}\left(  \xi\right)  q_{2,\xi
}\left(  -1\right)  -\hat{g}_{2}\left(  \xi\right)  =\hat{\gamma}_{-1}\left(
\xi\right)  ,\\
\hat{g}_{1}\left(  \xi\right)  q_{1,\xi}^{\prime}\left(  1\right)  +\left(
\Delta_{\Gamma_{0}}g_{2}\right)  ^{\symbol{94}}q_{2,\xi}^{\prime}\left(
1\right)  -\hat{g}_{2}\left(  \xi\right)  =\hat{\gamma}_{1}\left(  \xi\right)
.
\end{array}
\right.
\]
Using the symmetry of $q_{1,\xi}$ and $q_{2,\xi}$, we obtain
\[
\left\{
\begin{array}
[c]{l}%
\hat{g}_{1}\left(  \xi\right)  =\ \frac{\hat{\gamma}_{1}\left(  \xi\right)
-\hat{\gamma}_{-1}\left(  \xi\right)  }{2q_{1,\xi}^{\prime}\left(  1\right)
},\\
\left(  \Delta_{\Gamma_{0}}g_{2}\right)  ^{\symbol{94}}\left(  \xi\right)
=\frac{2\hat{g}_{2}\left(  \xi\right)  +\hat{\gamma}_{-1}\left(  \xi\right)
+\hat{\gamma}_{1}\left(  \xi\right)  }{2q_{2,\xi}^{\prime}\left(  1\right)  }.
\end{array}
\right.
\]
Taking inverse Fourier transform, using the analyticity and asymptotic
behavior of $q_{1,\xi}^{\prime}\left(  1\right)  $ and $q_{2,\xi}^{\prime
}\left(  1\right)  $ (it behaves like $\left\vert \xi\right\vert ^{-1}$ as
$\left\vert \xi\right\vert \rightarrow+\infty,$ see Lemma 16 of \cite{Liu2}%
)$,$ we get the desired weighted norm estimate. Note that in \cite{Liu2}, we
have considered the algebraically weighted norms, here we are dealing with
exponentially weighted norms.

We now turn to the general case $\vartheta\neq0.$ Let us use $\Phi^{\ast}$ to
denote the solution of the problem
\[
\left\{
\begin{array}
[c]{l}%
\partial_{t}^{2}\Phi^{\ast}+\partial_{l}^{2}\Phi^{\ast}+\frac{1}{l}%
\partial_{l}\Phi^{\ast}=\vartheta,\text{ }\left(  t,l\right)  \in\left[
-1,1\right]  \times\mathbb{[}0,+\infty).\\
\Phi^{\ast}\left(  \pm1,l\right)  =0.
\end{array}
\right.
\]
Note that $\Phi^{\ast}$ decays like $O\left(  e^{-\beta l}\right)  .$ Then we
can write $\Phi=\Phi^{\ast}+\tilde{\Phi}$ and proceed similarly as before.
\end{proof}

For function $\eta$ define on $\mathbb{R}^{+},$ we use the notation
\[
\left\Vert \eta\right\Vert _{s}:=\left\Vert \eta\right\Vert _{C^{2,\alpha
}\left(  \left[  s,+\infty\right]  \right)  },\text{ }\left\Vert
\eta\right\Vert _{\left[  c,d\right]  }:=\left\Vert \eta\right\Vert
_{C^{2,\alpha}\left(  \left[  c,d\right]  \right)  }.
\]
Moreover, for function $\tilde{\eta}$ define on $\left[  -1,1\right]
\times\mathbb{R}^{+},$ we set
\begin{align*}
\left\Vert \tilde{\eta}\right\Vert _{s,\symbol{94}}  &  :=\left\Vert
g\right\Vert _{C^{2,\alpha}\left(  \left[  -1,1\right]  \times\lbrack
s,+\infty\right)  },\\
\left\Vert \tilde{\eta}\right\Vert _{\left[  c,d\right]  ,\symbol{94}}  &
:=\left\Vert g\right\Vert _{C^{2,\alpha}\left(  \left[  -1,1\right]
\times\left[  c,d\right]  \right)  }.
\end{align*}

\begin{proof}
[Proof of Proposition \ref{uni}]Let $\beta>0$ be a fixed small constant. We
first observe that in the perturbation term $P,$ the term $2\nabla\zeta
_{a}\nabla\phi+\Delta\zeta_{a}\phi$ is compactly supported. On the other hand,
by the estimate (non-optimal) $\left\vert p^{\prime}\right\vert \leq C,$ and
the formula
\[
H_{\Gamma_{0}}=\frac{1}{l}\left(  \frac{lp^{\prime}}{\sqrt{1+p^{\prime2}}%
}\right)  ^{\prime},
\]
we find that the rest terms in $P$ can be estimated as
\[
\left(  I_{2}-I_{1}+\Delta_{\Gamma_{0}}\bar{W}-\Delta_{\Gamma_{s}}\bar
{W}\right)  \zeta_{a}=o\left(  \left\vert H_{\Gamma_{0}}\right\vert
+\left\vert f\right\vert +\left\vert \phi\right\vert +\left\vert f^{\prime
}\right\vert \right)  +O\left(  l^{-2}\right)  .
\]
We use $\phi^{\ast}$ to denote the solution of the problem
\[
\left\{
\begin{array}
[c]{l}%
\Delta\phi^{\ast}=2\nabla\zeta_{a}\nabla\phi+\Delta\zeta_{a}\phi,\\
\phi^{\ast}\left(  t,l\right)  =0,t=\pm1.
\end{array}
\right.
\]
Writing $\Psi_{a}$ as $\phi^{\ast}+\tilde{\Psi}_{a}$ \ and applying Lemma
\ref{estimate1}, we obtain%
\begin{equation}
\left\Vert \phi\right\Vert _{a+s,\symbol{94}}=O\left(  e^{-\beta s}\left\Vert
\phi\right\Vert _{\left[  a,a+1\right]  ,\symbol{94}}\right)  +O\left(
\left\Vert f\right\Vert _{a}+\left\Vert H_{\Gamma_{0}}\right\Vert
_{C^{0,\alpha}\left(  \left[  a,+\infty\right]  \right)  }+a^{-2}\right)
,s\geq0. \label{L1}%
\end{equation}
On the other hand, from equation $\left(  \ref{3-}\right)  $ and $\left(
\ref{3+}\right)  ,$ we get
\[
\gamma^{-}=O\left(  f^{2}+f^{\prime2}+\left(  \partial_{t}\phi\right)
^{2}\right)  \text{ and }\gamma^{+}=O\left(  \left(  \partial_{t}\phi\right)
^{2}\right)  +o\left(  f\right)  .
\]
Combining this with Lemma \ref{estimate2} and $\left(  \ref{L1}\right)  ,$ we
infer%
\begin{align}
&  \left\Vert f\right\Vert _{a+s}+\left\Vert H_{\Gamma_{0}}\right\Vert
_{C^{0,\alpha}\left(  \left[  a+s,+\infty\right]  \right)  }\nonumber\\
&  =O\left[  e^{-\beta s}\left(  \left\Vert f\right\Vert _{\left[
a,a+1\right]  }+\left\Vert \phi\right\Vert _{\left[  a,a+1\right]
,\symbol{94}}\right)  \right]  +O\left(  a^{-3}\right)  ,s\geq0. \label{L3}%
\end{align}

Now let us define
\[
\theta\left(  a\right)  :=\left\Vert f\right\Vert _{a}+\left\Vert
H_{\Gamma_{0}}\right\Vert _{C^{0,\alpha}\left(  \left[  a,+\infty\right]
\right)  }+\left\Vert \phi\right\Vert _{a,\symbol{94}}.
\]
Fix a constant $d$. From $\left(  \ref{L1}\right)  $ and $\left(
\ref{L3}\right)  ,$ we are led to%
\begin{equation}
\label{e7}\theta\left(  a+d\right)  \leq Ce^{-\beta d}\left(  \theta\left(
a\right)  +\theta\left(  a-d\right)  \right)  +Ca^{-2}.
\end{equation}
Applying this estimate at $a-dj,j=0,1,2,...,\left[  \frac{a-r_{0}}{d}\right]
$, we get%
\begin{align*}
\theta\left(  a\right)   &  \leq Ce^{-\beta d}\left(  \theta\left(
a-d\right)  +\theta\left(  a-2d\right)  \right)  +Ca^{-2}\\
&  \leq C^{2}e^{-\beta d}\left[  e^{-\beta d}\left(  \theta\left(
a-2d\right)  +\theta\left(  a-3d\right)  \right)  +a^{-2}\right] \\
&  +C^{2}e^{-\beta d}\left[  e^{-\beta d}\left(  \theta\left(  a-3d\right)
+\theta\left(  a-4d\right)  \right)  +a^{-2}\right]  +Ca^{-2}\\
&  \leq...\leq\bar{C}a^{-2},
\end{align*}
for another constant $\bar{C},$ provided that $d$ is sufficiently large. With
this decay information at hand, repeating the above arguments, we find that
$\left\vert f^{\prime}\left(  a\right)  \right\vert +\left\vert p^{\prime
}\left(  a\right)  \right\vert =O\left(  a^{-1}\right)  .$ Hence the term
$a^{-2}$ in $\left(  \ref{L1}\right)  $ can be improved to $a^{-3},$ and
$\left(  \ref{e7}\right)  $ can be refined to%
\[
\theta\left(  a+d\right)  \leq Ce^{-\beta d}\left(  \theta\left(  a\right)
+\theta\left(  a-d\right)  \right)  +Ca^{-3}.
\]
Similar arguments as before yield $\theta\left(  a\right)  \leq Ca^{-3},$
which in particular implies
\[
\left\Vert H_{\Gamma_{0}}\right\Vert _{C^{0,\alpha}\left(  \left[
a,+\infty\right]  \right)  }\leq Ca^{-3},\text{ for any }a>r_{0}.
\]
Hence the function $p_{k}$ satisfies%
\[
\left(  \frac{lp_{k}^{\prime}}{\sqrt{1+p_{k}^{\prime2}}}\right)  ^{\prime
}=O\left(  l^{-3}\right)  .
\]
Integrating this equation once, we get $p_{k}^{\prime}\left(  l\right)
=\frac{\bar{k}}{l}+O\left(  l^{-3}\right)  ,$ for some constant $\bar{k}.$
Necessarily $\bar{k}=k.$ Integrating once more, we get
\[
p_{k}\left(  l\right)  =k\ln l+b_{k}+O\left(  l^{-1}\right)  .
\]

To show that $\left\vert b_{k}\right\vert \leq C$, it remains to prove that
$\left\vert p_{k}\left(  r_{0}\right)  \right\vert \leq C.$ If this were not
true, then after suitable translation along the $z$ axis, a subsequence of
$W_{k}$ converges to a solution $w$ of the free boundary problem $\left(
\ref{EN}\right)  ,$ with $\left\{  \left(  r,z\right)  :\left\vert w\left(
r,z\right)  \right\vert <1\right\}  =\left\{  \left(  r,z\right)
:c_{1}<r<c_{2}<+\infty\right\}  .$ This is not possible. The proof is thus completed.
\end{proof}

\begin{lemma}
\label{k}As $k\rightarrow0,$ $W_{k}$ converges in $C_{loc}^{0,\alpha}$ to\ the
function $\mathcal{H}\left(  z-1\right)  $ in the upper half space.
\end{lemma}

\begin{proof}
Similar as Lemma \ref{dis}, we can prove that the distance of the free
boundary of $W_{k}$ to the origin is uniformly bounded for $k.$

Using Lemma \ref{uni}, we deduce that $W_{k}$ converges to a solution
$W_{\infty}.$ Since for each $W_{k},$ its free boundary point is
nondegenerated, the free boundary of $W_{\infty}$ is smooth away from the
origin, We use $z=f\left(  r\right)  $ to represent the curve%
\[
\partial\left\{  \left\vert W_{\infty}\right\vert <1\right\}  \cap\left\{
W_{\infty}=1\right\}  .
\]
The estimate in Lemma \ref{uni} tells us that $f$ is bounded. Since the
surface $z=f\left(  r\right)  $ has nonnegative mean curvature, with respect
to the unit normal pointing towards the positive $z$ direction, we get
\begin{equation}
\left[  \frac{rf^{\prime}\left(  r\right)  }{\sqrt{1+f^{\prime}\left(
r\right)  ^{2}}}\right]  ^{\prime}\geq0. \label{con}%
\end{equation}
On the other hand, we know that $f\left(  r\right)  \rightarrow C_{0}$ for
some constant $C_{0}.$ This together with $\left(  \ref{con}\right)  $ implies
that
\[
\lim_{r\rightarrow+\infty}f\left(  r\right)  \leq\max_{r\in\lbrack0,+\infty
)}f\left(  r\right)  .
\]
Then a sliding plane type argument using solutions of the form $\mathcal{H}%
\left(  z-\beta\right)  $ implies that $f$ is a constant and
\[
W_{\infty}\left(  z\right)  =\mathcal{H}\left(  z-\alpha\right)  ,
\]
for some $\alpha\geq1.$ This argument also tells us that for $k$ small,
$\partial\left\{  W_{k}<1\right\}  \cap\left\{  W_{k}=-1\right\}  $ intersects
with the $r$ axis and $\partial\left\{  W_{k}<1\right\}  \cap\left\{
W_{k}=1\right\}  $ does not intersect with the $r$ axis.

If $\alpha>1,$ then we could use a blow up argument for $W_{k}$ near the point
$\left(  0,\alpha-1\right)  $ to get a contradiction. This completes the proof.
\end{proof}

With all these preparation, we are in a position to prove Theorem \ref{main}
in $3D.$

\begin{proof}
[Proof of Theorem \ref{main} in $\mathbb{R}^{3}$]Using Lemma $\ref{k},$ we
would like to perform a blow up analysis on $W_{k}$ to get a solution to the
one phase free boundary problem. Indeed, let $\rho_{k}$ be the distance of the
free boundary (the part where $W_{k}=-1$) to the origin. Then by Lemma
\ref{k}, $\rho_{k}\rightarrow0$ as $k\rightarrow0.$ Let us define
\[
\psi_{k}\left(  X\right)  =\frac{W_{k}\left(  \rho_{k}X\right)  +1}{\rho_{k}%
}.
\]
Then $\psi_{k}$ converges to a solution $u$ of the one-phase free boundary
problem. This is the desired solution. The asymptotic behavior $\left(
\ref{lim}\right)  $ follows from the positivity of the mean curvature of the
free boundary.
\end{proof}

\begin{remark}
The Hauswirth-Helein-Pacard solution in 2D can also be constructed using our
variational and blow up technique. Note that in 2D, we need to construct
solutions to $\left(  \ref{EN}\right)  $ with nodal set asymptotic to straight
lines at infinity.
\end{remark}

\section{The proof of Theorem \ref{main} for dimension $n>3$\label{Sec6}}

In this section, we assume $n>3$. The proof is essentially same as before,
except that at some points we need to modify certain estimates.

As we already show in the $3D$ case, the construction of solutions to our free
boundary problem is closely related to the geometry of the catenoids. Let us
first of all recall the definition of catenoids in $\mathbb{R}^{n}$, which are
codimension one minimal submanifolds. Let $\phi$ be the solution of%
\begin{equation}
\left\{
\begin{array}
[c]{l}%
\frac{\phi^{\prime\prime}}{1+\phi^{\prime2}}-\frac{n-2}{\phi}=0,\\
\phi\left(  0\right)  =1,\phi^{\prime}\left(  0\right)  =0.
\end{array}
\right.  \label{c1}%
\end{equation}
Then the manifold given by $r:=\phi\left(  z\right)  $ is a minimal
submanifold, called catenoid. The principle curvatures are given by
\[
k_{1}=...=k_{n-2}=\frac{1}{\phi\left(  1+\phi^{\prime2}\right)  ^{\frac{1}{2}%
}},\text{ \ }k_{n-1}=-\frac{\phi^{\prime\prime}}{\left(  1+\phi^{\prime
2}\right)  ^{\frac{3}{2}}}.
\]
Introducing a parametrization:
\begin{equation}
r=\left(  \eta\left(  s\right)  \right)  ^{\frac{1}{n-2}},z=\int_{0}%
^{s}\left(  \eta\left(  t\right)  \right)  ^{\frac{3-n}{n-2}}dt,\label{para}%
\end{equation}
we find that $\eta$ satisfies%
\[
\frac{1}{n-2}\eta^{\prime\prime}\frac{1}{1+\left(  \frac{1}{n-2}\eta^{\prime
}\right)  ^{2}}-\frac{n-2}{\eta}=0.
\]
From this we get $\eta\left(  s\right)  =\cosh\left(  \left(  n-2\right)
s\right)  .$

In the upper $z$ space, we can also write this catenoid as
\begin{equation}
z=\bar{\phi}\left(  r\right)  ,r\in\lbrack1,+\infty).\label{fibar}%
\end{equation}
Then there are constants $c_{n},c_{n}^{\prime}$ such that as $r\rightarrow
+\infty,$
\[
\bar{\phi}\left(  r\right)  \sim c_{n}-c_{n}^{\prime}r^{3-n}.
\]
In terms of $\bar{\phi},$ the equation in $\left(  \ref{c1}\right)  $ can be
written as
\[
\frac{\bar{\phi}^{\prime\prime}\left(  r\right)  }{1+\bar{\phi}^{\prime
2}\left(  r\right)  }+\frac{\left(  n-2\right)  \bar{\phi}^{\prime}\left(
r\right)  }{r}=0.
\]
Note that for each $\rho>0,$ the rescaled function $z=\rho\bar{\phi}\left(
\rho^{-1}r\right)  $ also gives us a catenoid. We refer to \cite{Tam} for more
detailed properties on catenoids, including their Morse index.

For each $\alpha>0,$ we shall use $r=\phi_{\alpha}\left(  z\right)  $ to
represent the catenoid which satisfies $\phi_{\alpha}\left(  0\right)
=\alpha.$ This catenoid will also be written as $z=\bar{\phi}_{\alpha}\left(
r\right)  $. On the other hand, we use $z=\bar{\phi}_{\alpha}^{\ast}\left(
r\right)  $ to represent the catenoid with
\[
\lim_{r\rightarrow+\infty}\bar{\phi}_{\alpha}^{\ast}\left(  r\right)  =\alpha.
\]
This catenoid will also be written as $r=\phi_{\alpha}^{\ast}\left(  z\right)
.$

Let $k>1$ be a parameter. For each $a$ large, let
\[
\Omega_{a}:=\left\{  \left(  r,z\right)  :r\in\left[  0,a\right]  ,z\in\left[
0,b_{\varepsilon}\right]  \right\}  ,
\]
where $b_{\varepsilon}=\bar{\phi}_{k}^{\ast}\left(  a\right)  +2+\delta
_{\varepsilon}$ and $\delta_{\varepsilon}$ is defined by $\left(
\ref{delta}\right)  .$ Set $L_{a}:=L_{1,a}\cup L_{2,a},$ where
\[
L_{1,a}:=\left\{  \left(  a,z\right)  :z\in\left[  0,b_{\varepsilon}\right]
\right\}  ,\quad\mbox{and}\quad L_{2,a}:=\left\{  \left(  r,b_{\varepsilon
}\right)  :r\in\left[  0,a\right]  \right\}  .
\]
We then define a function $\omega=\omega\left(  r,z\right)  ,$ depending on
the parameter $\varepsilon$ and $a,$ to be
\[
\omega\left(  r,z\right)  =w_{\varepsilon,\bar{\phi}_{k}^{\ast}\left(
a\right)  -\varepsilon}\left(  z-\bar{\phi}_{k}^{\ast}\left(  a\right)
\right)  .
\]
Here, same as before, $w_{\varepsilon,l}$ is the function appeared in Lemma
\ref{sub}.

For $\varepsilon$ sufficiently small, we need to construct mountain pass
solutions for the problem
\[
\left\{
\begin{array}
[c]{l}%
-\partial_{r}^{2}u-\frac{1}{r}\partial_{r}u-\partial_{z}^{2}u+\frac{1}%
{2}F_{\varepsilon}^{\prime}\left(  u\right)  =0\text{ in }\Omega_{a},\\
\partial_{r}u\left(  0,z\right)  =0,\partial_{z}u\left(  r,0\right)  =0,\\
u=\omega\text{, on }L_{a}.
\end{array}
\right.
\]
Similarly as before, using parabolic flow, we can first of all construct two
solutions $u_{1},u_{2},$ where $u_{1}$ has almost horizontal nodal set and
$u_{2}$ has almost vertical nodal set. Moreover, $\partial_{r}u_{i}<0$ and
$\partial_{z}u_{i}>0$ in $\Omega_{a}.$ Furthermore, we can assume
\[
\int_{\Omega_{a}}\left(  \left\vert \nabla u_{i}\right\vert ^{2}%
+F_{\varepsilon}\left(  u_{i}\right)  \right)  \leq\frac{4}{n-1}%
a^{n-1}+o\left(  1\right)  ,
\]
where $o\left(  1\right)  $ is a term tending to $0$ as $\varepsilon
\rightarrow0.$

Let $\mathcal{E}$ be the set of $C^{1}$ functions $g$ satisfying the following properties:

\noindent(I) $u_{1}<g<u_{2}$ in $\Omega_{a}$,

\noindent(II) $\partial_{z}g>0;\partial_{r}g<0,$ in $\Omega_{a},$

\noindent(III) $g|_{L_{a}}=\omega,$

\noindent(IV)$\partial_{r}g\left(  0,z\right)  =0,\partial_{z}g\left(
r,0\right)  =0.$

The following geometric property of the catenoids will be used later on.

\begin{lemma}
Let $k>1$ be a fixed constant. Suppose $a$ is large. Then for each $c<\phi
_{k}^{\ast}\left(  0\right)  ,$
\[
\int_{c}^{a}\sqrt{1+\left(  \bar{\phi}_{c}^{\prime}\left(  r\right)  \right)
^{2}}r^{n-2}dr-\frac{a^{n-1}}{n-1}\geq\frac{\delta}{2}c^{n-1},
\]
where
\[
\delta=\frac{1}{2\left(  n-2\right)  }\left(  1-\left(  \frac{1}{2}\right)
^{\frac{1}{n-2}}\right)  .
\]

\end{lemma}

\begin{proof}
We first consider the case of $c=1.$ Let $s_{0}$ be the constant defined by
\[
\cosh\left(  \left(  n-2\right)  s_{0}\right)  =a^{n-2}.
\]
Using the parametrization $\left(  \ref{para}\right)  ,$ we compute
\begin{align*}
\int_{1}^{a}\sqrt{1+\bar{\phi}_{1}^{\prime2}\left(  r\right)  }r^{n-2}dr  &
=\int_{0}^{\bar{\phi}_{1}\left(  a\right)  }\sqrt{1+\phi_{1}^{\prime2}\left(
z\right)  }\phi_{1}^{n-2}\left(  z\right)  dz\\
&  =\int_{0}^{s_{0}}\sqrt{1+\sinh^{2}\left(  \left(  n-2\right)  s\right)
}\cosh^{\frac{1}{n-2}}\left(  \left(  n-2\right)  s\right)  ds\\
&  =\frac{1}{n-2}\int_{0}^{\sinh\left(  \left(  n-2\right)  s_{0}\right)
}\left(  1+x^{2}\right)  ^{\frac{1}{2\left(  n-2\right)  }}dx\\
&  \geq\frac{1}{n-2}\int_{\frac{1}{2}}^{\sinh\left(  \left(  n-2\right)
s_{0}\right)  }x^{\frac{1}{n-2}}dx+\frac{1}{2\left(  n-2\right)  }.
\end{align*}
It follows that
\begin{align*}
&  \int_{1}^{a}\sqrt{1+\bar{\phi}_{1}^{\prime2}\left(  r\right)  }r^{n-2}dr\\
&  \geq\frac{1}{n-1}a^{n-1}+\delta+O\left(  a^{3-n}\right)  ,
\end{align*}
where $\delta=\frac{1}{2\left(  n-2\right)  }\left(  1-\left(  \frac{1}%
{2}\right)  ^{\frac{1}{n-2}}\right)  .$

Now since $\bar{\phi}_{c}\left(  r\right)  =c\bar{\phi}_{1}\left(
c^{-1}r\right)  ,$ \ we have
\begin{align*}
\int_{c}^{a}\sqrt{1+\bar{\phi}_{c}^{\prime2}\left(  r\right)  }r^{n-2}dr  &
=\int_{c}^{a}\sqrt{1+\bar{\phi}_{1}^{\prime2}\left(  c^{-1}r\right)  }%
r^{n-2}dr\\
&  =c^{n-1}\int_{1}^{c^{-1}a}\sqrt{1+\bar{\phi}_{1}^{\prime2}\left(  r\right)
}r^{n-2}dr\\
&  \geq c^{n-1}\left(  \frac{1}{n-1}c^{1-n}a^{n-1}+\delta+O\left(
a^{3-n}c^{n-3}\right)  \right) \\
&  =\frac{1}{n-1}a^{n-1}+\frac{\delta}{2}c^{n-1}.
\end{align*}
This is the desired estimate.
\end{proof}

\begin{lemma}
\label{a}Let $a$ be a large constant. There exists $\varepsilon_{0}>0$
depending on $a,$ such that for $\varepsilon<\varepsilon_{0},$ the following
is true: Suppose $\xi\in\mathcal{E}$ and $\xi\left(  \phi_{k-1}\left(
0\right)  ,0\right)  =-1+\varepsilon.$ Then
\[
\int_{\Omega_{a}}\left(  \left\vert \nabla\xi\right\vert ^{2}+F_{\varepsilon
}\left(  \xi\right)  \right)  \geq\frac{4}{n-1}a^{n-1}+\delta_{0},
\]
where $\delta_{0}>0$ is a constant independent of $\xi$ and $\varepsilon.$
\end{lemma}

\begin{proof}
We still use $A\left(  s\right)  $ to be denote the area of the surface
$\left\{  \left(  r,z\right)  :\xi\left(  r,z\right)  =s\right\}  .$ Consider
the points $P_{1},P_{2}$ whose $\left(  r,z\right)  $ coordinates are given by
$\left(  q,0\right)  $ and $\left(  q,\frac{k+1}{2}\right)  $ respectively,
where $q=\frac{1}{4}\phi_{k-1}\left(  0\right)  $. Let $\sigma>0$ be a small
positive constant(indepedent of $\varepsilon$) to be determined later on.
There are two possibilities.

Case 1. $\xi\left(  P_{1}\right)  >-1+\sigma$ or $\xi\left(  P_{2}\right)
>0.$

Subcase 1. $\xi\left(  P_{1}\right)  >-1+\sigma.$

Using Lemma \ref{a} and the fact that the catenoid is a minimal surface, we
find that for $s\in\left(  -1+\varepsilon,\xi\left(  P_{1}\right)  \right)  ,$%
\[
A\left(  s\right)  \geq\frac{1}{n-1}a^{n-1}+\frac{\delta}{2}q^{n-1}.
\]
Hence by the coerea formula we get
\begin{align}
\int_{\Omega_{a}}\left(  \left\vert \nabla\xi\right\vert ^{2}+F_{\varepsilon
}\left(  \xi\right)  \right)   &  \geq2\int_{-1}^{1}A\left(  s\right)
\sqrt{F_{\varepsilon}\left(  s\right)  }ds.\nonumber\\
&  \geq\left(  \int_{\xi\left(  P_{1}\right)  }^{1}+\int_{-1}^{\xi\left(
P_{1}\right)  }\right)  \left(  A\left(  s\right)  \sqrt{F_{\varepsilon
}\left(  s\right)  }\right)  ds\nonumber\\
&  \geq\frac{a^{n-1}}{n-1}\left(  1-\xi\left(  P_{1}\right)  \right)  \left(
1+O\left(  \varepsilon\right)  \right)  \nonumber\\
&  +\left(  \frac{a^{n-1}}{n-1}+\frac{\delta}{2}q^{n-1}\right)  \left(
\xi\left(  P_{1}\right)  -1\right)  \left(  1+O\left(  \varepsilon\right)
\right)  \nonumber\\
&  \geq\frac{1}{n-1}a^{n-1}+\frac{\delta\sigma}{2}q^{n-1}+O\left(
\varepsilon\right)  .\label{eq2}%
\end{align}
provided that $\varepsilon$ is sufficiently small.

Subcase 2. $\xi\left(  P_{2}\right)  >0.$

Similarly as Subcase 1, from Lemma \ref{a}, we deduce that for $s\in\left(
\frac{1-k}{5},0\right)  ,$%
\[
A\left(  s\right)  \geq\frac{1}{n-1}a^{n-1}+\frac{\delta}{2}q^{n-1}.
\]
Using this lower bound, under the assumption that $\varepsilon$ is small, we
can estimate
\begin{align}
\int_{\Omega_{a}}\left(  \left\vert \nabla\xi\right\vert ^{2}+F_{\varepsilon
}\left(  \xi\right)  \right)   &  \geq\left(  \int_{-1}^{\frac{1-k}{5}}%
+\int_{\frac{1-k}{5}}^{0}+\int_{0}^{1}\right)  \left(  A\left(  s\right)
\sqrt{F_{\varepsilon}\left(  s\right)  }\right)  ds\nonumber\\
&  \geq\frac{2}{n-1}a^{n-1}+\frac{\delta\left(  k-1\right)  }{10}%
q^{n-1}+O\left(  \varepsilon\right)  .\label{eq1}%
\end{align}

Case 2. $\xi\left(  P_{1}\right)  +1<\sigma$ and $\xi\left(  P_{2}\right)
<0.$

Let us define
\begin{align*}
\Omega_{1}^{\ast} &  =\left\{  \left(  r,z\right)  \in\Omega_{a}:r>\phi
_{k-1}\left(  0\right)  \right\}  ,\\
\Omega_{2}^{\ast} &  =\left\{  \left(  r,z\right)  \in\Omega_{a}:r<\phi
_{k-1}\left(  0\right)  \right\}  .
\end{align*}
Then the energy in the region $\Omega_{1}^{\ast}$ has the estimate%
\begin{align}
\int_{\Omega_{1}^{\ast}}\left(  \left\vert \nabla\xi\right\vert ^{2}%
+F_{\varepsilon}\left(  \xi\right)  \right)   &  \geq\int_{\phi_{k-1}\left(
0\right)  }^{a}\left(  \int_{0}^{b_{\varepsilon}}\left(  \left(  \partial
_{z}\xi\right)  ^{2}+F_{\varepsilon}\left(  \xi\right)  \right)  dz\right)
r^{n-2}dr\nonumber\\
&  \geq\frac{4+O\left(  \varepsilon\right)  }{n-1}\left(  a^{n-1}-\left(
\phi_{k-1}\left(  0\right)  \right)  ^{n-1}\right)  .\label{es1}%
\end{align}
On the other hand, for $r\in\left(  q,\phi_{k-1}\left(  0\right)  \right)  ,$
using the fact that $\xi\left(  r,\frac{k+1}{2}\right)  <0$ and
\[
-1+\varepsilon<\xi\left(  r,0\right)  <-1+\sigma,
\]
we obtain%
\begin{equation}
\int_{0}^{b_{\varepsilon}}\left(  \left(  \partial_{z}\xi\right)
^{2}+F_{\varepsilon}\left(  \xi\right)  \right)  dz\geq2+\left(  \left(
\frac{2\left(  -1+\sigma\right)  }{k+1}\right)  ^{2}+1\right)  \frac{k+1}%
{2}+O\left(  \varepsilon\right)  .\label{1}%
\end{equation}
Then we can choose a small constant $\sigma>0$ such that the right hand side
of $\left(  \ref{1}\right)  $ is bounded below by a constant $4+\delta_{1},$
where $\delta_{1}>0$ is independent of $\varepsilon.$ Then we can estimate
\begin{align}
\int_{\Omega_{2}^{\ast}}\left(  \left\vert \nabla\xi\right\vert ^{2}%
+F_{\varepsilon}\left(  \xi\right)  \right)   &  \geq\int_{q}^{\phi
_{k-1}\left(  0\right)  }\left(  \int_{0}^{b_{\varepsilon}}\left(  \left(
\partial_{z}\xi\right)  ^{2}+F_{\varepsilon}\left(  \xi\right)  \right)
dz\right)  r^{n-2}dr\nonumber\\
&  \geq\frac{4+\delta_{2}}{n-1}\left(  \left(  \phi_{k-1}\left(  0\right)
\right)  ^{n-1}-q^{n-1}\right)  .\label{es2}%
\end{align}
Combining $\left(  \ref{es1}\right)  $ and $\left(  \ref{es2}\right)  ,$ we
deduce that when $\varepsilon$ is sufficiently small,
\begin{align}
\int_{\Omega_{a}}\left(  \left\vert \nabla\xi\right\vert ^{2}+F_{\varepsilon
}\left(  \xi\right)  \right)   &  \geq\frac{4}{n-1}\left(  a^{n-1}-\left(
\phi_{k-1}\left(  0\right)  \right)  ^{n-1}\right)  \nonumber\\
&  +\frac{4+\delta_{1}}{n-1}\left(  \left(  \phi_{k-1}\left(  0\right)
\right)  ^{n-1}-q^{n-1}\right)  +O\left(  \varepsilon\right)  \nonumber\\
&  \geq\frac{4a^{n-1}}{n-1}+\frac{\delta_{1}}{2\left(  n-1\right)  }\left(
1-\frac{1}{4^{n-1}}\right)  \left(  \phi_{k-1}\left(  0\right)  \right)
^{n-1}.\label{eq3}%
\end{align}
From equations $\left(  \ref{eq2}\right)  ,\left(  \ref{eq1}\right)  $ and
$\left(  \ref{eq3}\right)  ,$ we conclude the proof.
\end{proof}

For each fixed $k>0$ and large $a$, when $\varepsilon$ is small, with the help
of Lemma \ref{a} and the parabolic flow, we then get a family of mountain pass
type solutions $U_{\varepsilon,a}$(depending on $k$)$,$ with the energy
estimate
\begin{equation}
\frac{4}{n-1}a^{n-1}+\delta_{0}\leq\int_{\Omega_{a}}\left(  \left\vert \nabla
U_{\varepsilon,a}\right\vert ^{2}+F_{\varepsilon}\left(  U_{\varepsilon
,a}\right)  \right)  \leq\frac{4}{n-1}a^{n-1}+C,\label{Ener}%
\end{equation}
for some constant $C$ independent of $\varepsilon,a.$

Letting $\varepsilon\rightarrow0,$ up to a subsequence, $U_{\varepsilon,a}$
converges to a function $V_{a}$ solving%
\[
\left\{
\begin{array}
[c]{c}%
\Delta V_{a}=0,\text{ in }\Omega_{a}\cap\left\{  \left\vert V_{a}\right\vert
<1\right\}  ,\\
\left\vert \nabla V_{a}\right\vert =1,\text{ on }\Omega_{a}\cap\partial
\left\{  \left\vert V_{a}\right\vert <1\right\}  .
\end{array}
\right.
\]
Moreover, on $\partial\Omega_{a},$ $V_{a}$ satisfies the boundary condition
inherited from $U_{\varepsilon,a}$.

As $a$ tends to infinity, up to a subsequence, $V_{a}$ converges to a solution
$W$ of the free boundary problem%
\[
\left\{
\begin{array}
[c]{c}%
\Delta W=0,\text{ in }\left\{  \left\vert W\right\vert <1\right\}  ,\\
\left\vert \nabla W\right\vert =1,\text{ on }\partial\left\{  \left\vert
W\right\vert <1\right\}  .
\end{array}
\right.
\]
The next lemma states that $W$ behaves like a catenoid at infinity.

\begin{lemma}
Let $\Omega=\left\{  \left(  r,z\right)  :z>0\text{ and }\left\vert W\left(
r,z\right)  \right\vert <1\right\}  .$ Let $r_{0}$ be a large constant.
Suppose that in the region where $r>r_{0},$
\begin{align*}
\partial\Omega\cap\left\{  \left(  r,z\right)  :W\left(  r,z\right)
=1\right\}   &  =\left\{  \left(  r,z\right)  :z=f_{1}\left(  r\right)
\right\}  ,\\
\partial\Omega\cap\left\{  \left(  r,z\right)  :W\left(  r,z\right)
=-1\right\}   &  =\left\{  \left(  r,z\right)  :z=f_{2}\left(  r\right)
\right\}  .
\end{align*}
Then there exists $k^{\prime}\geq1$ such that
\begin{align*}
f_{1}\left(  r\right)  -k^{\prime}-1 &  \rightarrow0,\\
f_{2}\left(  r\right)  -k^{\prime}+1 &  \rightarrow0,
\end{align*}
as $r\rightarrow+\infty.$
\end{lemma}

\begin{proof}
The mean curvatures of the surfaces $z=f_{1}\left(  r\right)  $ and
$z=f_{2}\left(  r\right)  $ have a sign. That is,
\[
\left(  \frac{r^{n-2}f_{1}^{\prime}}{\sqrt{1+f_{1}^{\prime2}}}\right)
^{\prime}\geq0,
\]
and
\[
\left(  \frac{r^{n-2}f_{2}^{\prime}}{\sqrt{1+f_{2}^{\prime2}}}\right)
^{\prime}\leq0.
\]
Then the proof of this lemma is similar as that of Lemma \ref{Asy}.
\end{proof}

Our next purpose is to show that $W$ has the desired asymptotic behavior, that
is, $k^{\prime}=k.$ To prove this, we need the following lemma, which is a
result parallel to Lemma \ref{y}.

\begin{lemma}
\label{a2}Suppose $k\neq k^{\prime}.$ Assume $A$ is large and $A<a.$ Let
$\xi=\xi\left(  r\right)  $ be a $C^{1}$ monotone increasing function
satisfying $\xi\left(  A\right)  =k^{\prime}$ and $\xi\left(  a\right)  =k.$
Then
\[
\int_{A}^{a}\sqrt{1+\left(  \xi^{\prime}\left(  r\right)  \right)  ^{2}%
}r^{n-2}dr\geq\frac{a^{n-1}-A^{n-1}}{n-1}+\frac{1}{2}\sqrt{A}\left(
k-k^{\prime}\right)  .
\]

\end{lemma}

\begin{proof}
We compute
\[
\int_{A}^{a}\left(  \sqrt{1+\left(  \xi^{\prime}\left(  r\right)  \right)
^{2}}-1\right)  r^{n-2}dr=\int_{A}^{a}\frac{\xi^{\prime}\left(  r\right)
^{2}}{\sqrt{1+\left(  \xi^{\prime}\left(  r\right)  \right)  ^{2}}+1}%
r^{n-2}dr.
\]
Let $S\subset\left[  A,a\right]  $ be the set where%
\[
\left\vert \frac{\xi^{\prime}\left(  r\right)  }{\sqrt{1+\left(  \xi^{\prime
}\left(  r\right)  \right)  ^{2}}+1}r^{n-\frac{5}{2}}\right\vert \leq1.
\]
Then using the fact that $A$ is large, which implies that in $S,$ $\xi
^{\prime}$ is small, we obtain%
\[
\left\vert \int_{S}\xi^{\prime}\left(  r\right)  dr\right\vert \leq\frac
{8}{2n-7}A^{\frac{7}{2}-n}.
\]
Therefore, when $A$ is sufficiently large, from $\int_{\left[  A,a\right]
}\xi^{\prime}\left(  r\right)  dr=k-k^{\prime},$ we get
\begin{align*}
&  \int_{\left[  A,a\right]  \backslash S}\frac{\xi^{\prime}\left(  r\right)
^{2}}{\sqrt{1+\left(  \xi^{\prime}\left(  r\right)  \right)  ^{2}}+1}%
r^{n-2}dr\\
&  \geq\int_{\left[  A,a\right]  \backslash S}\xi^{\prime}\left(  r\right)
r^{\frac{1}{2}}dr\\
&  \geq\sqrt{A}\int_{\left[  A,a\right]  \backslash S}\xi^{\prime}\left(
r\right)  dr\\
&  \geq\frac{1}{2}\sqrt{A}\left(  k-k^{\prime}\right)  ,
\end{align*}
This implies that
\[
\int_{A}^{a}\left(  \sqrt{1+\left(  \xi^{\prime}\left(  r\right)  \right)
^{2}}-1\right)  r^{n-2}dr\geq\frac{1}{2}\sqrt{A}\left(  k-k^{\prime}\right)  .
\]
The proof is then completed.
\end{proof}

With Lemma \ref{a2} at hand, we can use the energy upper bound $\left(
\ref{Ener}\right)  $ and proceed similarly as the proof of Proposition
\ref{Asym}, to conclude that for the solution $W,$ there holds $k^{\prime}=k.$
That is, the nodal set of $W$ is asymptotic to $z=k$ at infinity. Denote this
solution by $W_{k}.$

The next step is to analyze the precise asymptotic behavior of $W_{k},$
uniformly in $k$ as $k\rightarrow0.$ This can be achieved from similar
arguments as that of Section \ref{Sec5}, with straightforward
modifications(The decay rate of the principle curvatures are different).
Finally, same as the $3$D case, suitable blow up sequence of $W_{k}$ near the
origin then converges to a desired solution of the one phase free boundary problem.

\end{document}